\documentclass{amsart}

\usepackage{amsfonts}
\usepackage{amsmath}
\usepackage{amssymb}
\usepackage{latexsym}
\usepackage{amscd}
\usepackage{amsthm}
\usepackage{subfigure}
\usepackage{graphicx}
\usepackage[all]{xy}
\usepackage{color}

\newtheorem{thm}{Theorem}[section]
\newtheorem{prop}[thm]{Proposition}
\newtheorem{lem}[thm]{Lemma}

\theoremstyle{definition}

\newtheorem{rem}[thm]{Remark}

\newcommand{\M}{\mathcal{M}}
\newcommand{\PM}{\mathcal{PM}}

\begin{document}

\title{Infinite presentations for fundamental groups of surfaces}
%\subjclass{Primary 57M05; Secondary 20F05, 57M07, 20F65}
\author[R. Kobayashi]{Ryoma Kobayashi}
\address{Department of General Education\\
National Institute of Technology, Ishikawa College\\
Tsubata Ishikawa 929-0392 JAPAN}
\email{kobayashi\_ryoma@ishikawa-nct.ac.jp}
\thanks{The author is supported by JSPS KAKENHI Grant Number JP19K14542 and JP22K13920}

\maketitle

%%%%%%%%%%%%%%%%%%%%%%%%%%%%%%%%%%%%%%%%%%%%%%%%%%%%%%%%%%%%%%%%%%%%%%%%%%%%%%%%%%%%%%%%%%%%%%%%%%%%
%%%%%%%%%%%%%%%%%%%%%%%%%%%%%%%%%%%%%%%%%%%%%%%%%%%%%%%%%%%%%%%%%%%%%%%%%%%%%%%%%%%%%%%%%%%%%%%%%%%%
%%%%%%%%%%%%%%%%%%%%%%%%%%%%%%%%%%%%%%%%%%%%%%%%%%%%%%%%%%%%%%%%%%%%%%%%%%%%%%%%%%%%%%%%%%%%%%%%%%%%
%%%%%%%%%%%%%%%%%%%%%%%%%%%%%%%%%%%%%%%%%%%%%%%%%%%%%%%%%%%%%%%%%%%%%%%%%%%%%%%%%%%%%%%%%%%%%%%%%%%%
%%%%%%%%%%%%%%%%%%%%%%%%%%%%%%%%%%%%%%%%%%%%%%%%%%%%%%%%%%%%%%%%%%%%%%%%%%%%%%%%%%%%%%%%%%%%%%%%%%%%
%%%%%%%%%%%%%%%%%%%%%%%%%%%%%%%%%%%%%%%%%%%%%%%%%%%%%%%%%%%%%%%%%%%%%%%%%%%%%%%%%%%%%%%%%%%%%%%%%%%%
%%%%%%%%%%%%%%%%%%%%%%%%%%%%%%%%%%%%%%%%%%%%%%%%%%%%%%%%%%%%%%%%%%%%%%%%%%%%%%%%%%%%%%%%%%%%%%%%%%%%
%%%%%%%%%%%%%%%%%%%%%%%%%%%%%%%%%%%%%%%%%%%%%%%%%%%%%%%%%%%%%%%%%%%%%%%%%%%%%%%%%%%%%%%%%%%%%%%%%%%%
%%%%%%%%%%%%%%%%%%%%%%%%%%%%%%%%%%%%%%%%%%%%%%%%%%%%%%%%%%%%%%%%%%%%%%%%%%%%%%%%%%%%%%%%%%%%%%%%%%%%
%%%%%%%%%%%%%%%%%%%%%%%%%%%%%%%%%%%%%%%%%%%%%%%%%%%%%%%%%%%%%%%%%%%%%%%%%%%%%%%%%%%%%%%%%%%%%%%%%%%%
\begin{abstract}
For any finite type connected surface $S$, we give an infinite presentation of the fundamental group $\pi_1(S,\ast)$ of $S$ based at an interior point $\ast\in{S}$ whose generators are represented by simple loops.
When $S$ is non-orientable, we also give an infinite presentation of the subgroup of $\pi_1(S,\ast)$ generated by elements which are represented by simple loops whose regular neighborhoods are annuli.
\end{abstract}

%%%%%%%%%%%%%%%%%%%%%%%%%%%%%%%%%%%%%%%%%%%%%%%%%%%%%%%%%%%%%%%%%%%%%%%%%%%%%%%%%%%%%%%%%%%%%%%%%%%%
%%%%%%%%%%%%%%%%%%%%%%%%%%%%%%%%%%%%%%%%%%%%%%%%%%%%%%%%%%%%%%%%%%%%%%%%%%%%%%%%%%%%%%%%%%%%%%%%%%%%
%%%%%%%%%%%%%%%%%%%%%%%%%%%%%%%%%%%%%%%%%%%%%%%%%%%%%%%%%%%%%%%%%%%%%%%%%%%%%%%%%%%%%%%%%%%%%%%%%%%%
%%%%%%%%%%%%%%%%%%%%%%%%%%%%%%%%%%%%%%%%%%%%%%%%%%%%%%%%%%%%%%%%%%%%%%%%%%%%%%%%%%%%%%%%%%%%%%%%%%%%
%%%%%%%%%%%%%%%%%%%%%%%%%%%%%%%%%%%%%%%%%%%%%%%%%%%%%%%%%%%%%%%%%%%%%%%%%%%%%%%%%%%%%%%%%%%%%%%%%%%%
%%%%%%%%%%%%%%%%%%%%%%%%%%%%%%%%%%%%%%%%%%%%%%%%%%%%%%%%%%%%%%%%%%%%%%%%%%%%%%%%%%%%%%%%%%%%%%%%%%%%
%%%%%%%%%%%%%%%%%%%%%%%%%%%%%%%%%%%%%%%%%%%%%%%%%%%%%%%%%%%%%%%%%%%%%%%%%%%%%%%%%%%%%%%%%%%%%%%%%%%%
%%%%%%%%%%%%%%%%%%%%%%%%%%%%%%%%%%%%%%%%%%%%%%%%%%%%%%%%%%%%%%%%%%%%%%%%%%%%%%%%%%%%%%%%%%%%%%%%%%%%
%%%%%%%%%%%%%%%%%%%%%%%%%%%%%%%%%%%%%%%%%%%%%%%%%%%%%%%%%%%%%%%%%%%%%%%%%%%%%%%%%%%%%%%%%%%%%%%%%%%%
%%%%%%%%%%%%%%%%%%%%%%%%%%%%%%%%%%%%%%%%%%%%%%%%%%%%%%%%%%%%%%%%%%%%%%%%%%%%%%%%%%%%%%%%%%%%%%%%%%%%
\section{Introduction}

For any surface $S$ and any point $\ast$ in the interior of $S$, let $\pi_1(S,\ast)$ denote the fundamental group of $S$ based at $\ast$.
When $S$ is non-orientable, we denote by $\pi_1^+(S,\ast)$ the subgroup of $\pi_1(S,\ast)$ generated by elements which are represented by simple loops whose regular neighborhoods are annuli, called \textit{two-sided simple loops}.
A presentation of $\pi_1(S,\ast)$ is well known.
In particular,  $\pi_1(S,\ast)$, and also $\pi_1^+(S,\ast)$, are free groups if $S$ has a boundary.
For a connected closed orientable surface $S$, Putman \cite{P} gave an infinite presentation of $\pi_1(S,\ast)$.
In this paper we give infinite presentations of $\pi_1(S,\ast)$ and $\pi_1^+(S,\ast)$ whose generators are represented by simple loops, for any finite type connected surface $S$, as follows.

%%%%%%%%%%%%%%%%%%%%%%%%%%%%%%%%%%%%%%%%%%%%%%%%%%%%%%%%%%%%%%%%%%%%%%%%%%%%%%%%%%%%%%%%%%%%%%%%%%%%
%%%%%%%%%%%%%%%%%%%%%%%%%%%%%%%%%%%%%%%%%%%%%%%%%%%%%%%%%%%%%%%%%%%%%%%%%%%%%%%%%%%%%%%%%%%%%%%%%%%%
\begin{thm}\label{main-1}
For any finite type connected surface $S$, let $\pi$ be the group generated by symbols $S_\alpha$ for $\alpha\in\pi_1(S,\ast)$ which is represented by a non-trivial simple loop, and with the defining relations
\begin{enumerate}
\item	$S_{\alpha^{-1}}=S_\alpha^{-1}$,
\item	$S_{\alpha}S_{\beta}=S_{\gamma}$ if $\alpha\beta=\gamma$.
\end{enumerate}
Then $\pi$ is isomorphic to $\pi_1(S,\ast)$.
\end{thm}
%%%%%%%%%%%%%%%%%%%%%%%%%%%%%%%%%%%%%%%%%%%%%%%%%%%%%%%%%%%%%%%%%%%%%%%%%%%%%%%%%%%%%%%%%%%%%%%%%%%%
%%%%%%%%%%%%%%%%%%%%%%%%%%%%%%%%%%%%%%%%%%%%%%%%%%%%%%%%%%%%%%%%%%%%%%%%%%%%%%%%%%%%%%%%%%%%%%%%%%%%

%%%%%%%%%%%%%%%%%%%%%%%%%%%%%%%%%%%%%%%%%%%%%%%%%%%%%%%%%%%%%%%%%%%%%%%%%%%%%%%%%%%%%%%%%%%%%%%%%%%%
%%%%%%%%%%%%%%%%%%%%%%%%%%%%%%%%%%%%%%%%%%%%%%%%%%%%%%%%%%%%%%%%%%%%%%%%%%%%%%%%%%%%%%%%%%%%%%%%%%%%
\begin{thm}\label{main-2}
For any finite type connected non-orientable surface $S$, let $\pi^+$ be the group generated by symbols $S_\alpha$ for $\alpha\in\pi_1^+(S,\ast)$ which is represented by a non-trivial simple loop, and with the defining relations
\begin{enumerate}
\item	$S_{\alpha^{-1}}=S_\alpha^{-1}$,
\item	$S_{\alpha}S_{\beta}=S_{\gamma}$ if $\alpha\beta=\gamma$,
\item	$S_{\alpha}S_{\beta}S_{\alpha}^{-1}=S_\gamma$ if $\alpha\beta\alpha^{-1}=\gamma$.
\end{enumerate}
Then $\pi^+$ is isomorphic to $\pi_1^+(S,\ast)$.
\end{thm}
%%%%%%%%%%%%%%%%%%%%%%%%%%%%%%%%%%%%%%%%%%%%%%%%%%%%%%%%%%%%%%%%%%%%%%%%%%%%%%%%%%%%%%%%%%%%%%%%%%%%
%%%%%%%%%%%%%%%%%%%%%%%%%%%%%%%%%%%%%%%%%%%%%%%%%%%%%%%%%%%%%%%%%%%%%%%%%%%%%%%%%%%%%%%%%%%%%%%%%%%%

These results are useful in studies on the mapping class group of $S$ and its subgroups. 
For example, in \cite{KO1}, Theorem~\ref{main-2} is used to obtain an infinite presentation for the twist subgroup of the mapping class group of a compact non-orientable surface.

In order to prove Theorems~\ref{main-1} and \ref{main-2}, we use the following lemma.

%%%%%%%%%%%%%%%%%%%%%%%%%%%%%%%%%%%%%%%%%%%%%%%%%%%%%%%%%%%%%%%%%%%%%%%%%%%%%%%%%%%%%%%%%%%%%%%%%%%%
%%%%%%%%%%%%%%%%%%%%%%%%%%%%%%%%%%%%%%%%%%%%%%%%%%%%%%%%%%%%%%%%%%%%%%%%%%%%%%%%%%%%%%%%%%%%%%%%%%%%
\begin{lem}[cf. \cite{P}]\label{main-lem}
Let $G$ and $H$ be groups generated by sets $X$ and $Y$, respectively, such that $H$ acts on $G$.
Suppose that $X^\prime\subset{X}$ satisfies the following conditions.
\begin{itemize}
\item	$H(X^\prime)=X$.
\item	For any $x\in{X^\prime}$ and $y\in{Y}$, $y^{\pm1}(x)$ is in the subgroup of $G$ generated by $X^\prime$.
\end{itemize}
Then $X^\prime$ generates $G$.
\end{lem}
%%%%%%%%%%%%%%%%%%%%%%%%%%%%%%%%%%%%%%%%%%%%%%%%%%%%%%%%%%%%%%%%%%%%%%%%%%%%%%%%%%%%%%%%%%%%%%%%%%%%
%%%%%%%%%%%%%%%%%%%%%%%%%%%%%%%%%%%%%%%%%%%%%%%%%%%%%%%%%%%%%%%%%%%%%%%%%%%%%%%%%%%%%%%%%%%%%%%%%%%%

As groups acting on $\pi$ and $\pi^+$, we consider the \textit{pure mapping class group} of $S$.
Using this lemma, we show that $\pi$ and $\pi^+$ are generated by symbols corresponding to basic generators of $\pi_1(S,\ast)$ and $\pi_1^+(S,\ast)$, respectively.

In Section~\ref{mcg}, we define mapping class groups and pure mapping class groups of surfaces, and explain their generators.
In Sections~\ref{pi} and \ref{pi^+}, we prove Theorems~\ref{main-1} and \ref{main-2}, respectively.

Throughout this paper, we do not distinguish a loop from its homotopy class.

%%%%%%%%%%%%%%%%%%%%%%%%%%%%%%%%%%%%%%%%%%%%%%%%%%%%%%%%%%%%%%%%%%%%%%%%%%%%%%%%%%%%%%%%%%%%%%%%%%%%
%%%%%%%%%%%%%%%%%%%%%%%%%%%%%%%%%%%%%%%%%%%%%%%%%%%%%%%%%%%%%%%%%%%%%%%%%%%%%%%%%%%%%%%%%%%%%%%%%%%%
%%%%%%%%%%%%%%%%%%%%%%%%%%%%%%%%%%%%%%%%%%%%%%%%%%%%%%%%%%%%%%%%%%%%%%%%%%%%%%%%%%%%%%%%%%%%%%%%%%%%
%%%%%%%%%%%%%%%%%%%%%%%%%%%%%%%%%%%%%%%%%%%%%%%%%%%%%%%%%%%%%%%%%%%%%%%%%%%%%%%%%%%%%%%%%%%%%%%%%%%%
%%%%%%%%%%%%%%%%%%%%%%%%%%%%%%%%%%%%%%%%%%%%%%%%%%%%%%%%%%%%%%%%%%%%%%%%%%%%%%%%%%%%%%%%%%%%%%%%%%%%
%%%%%%%%%%%%%%%%%%%%%%%%%%%%%%%%%%%%%%%%%%%%%%%%%%%%%%%%%%%%%%%%%%%%%%%%%%%%%%%%%%%%%%%%%%%%%%%%%%%%
%%%%%%%%%%%%%%%%%%%%%%%%%%%%%%%%%%%%%%%%%%%%%%%%%%%%%%%%%%%%%%%%%%%%%%%%%%%%%%%%%%%%%%%%%%%%%%%%%%%%
%%%%%%%%%%%%%%%%%%%%%%%%%%%%%%%%%%%%%%%%%%%%%%%%%%%%%%%%%%%%%%%%%%%%%%%%%%%%%%%%%%%%%%%%%%%%%%%%%%%%
%%%%%%%%%%%%%%%%%%%%%%%%%%%%%%%%%%%%%%%%%%%%%%%%%%%%%%%%%%%%%%%%%%%%%%%%%%%%%%%%%%%%%%%%%%%%%%%%%%%%
%%%%%%%%%%%%%%%%%%%%%%%%%%%%%%%%%%%%%%%%%%%%%%%%%%%%%%%%%%%%%%%%%%%%%%%%%%%%%%%%%%%%%%%%%%%%%%%%%%%%
\section{On mapping class groups of surfaces}\label{mcg}

For $g\geq0$ and $m\geq0$, let $\Sigma_{g,m}$ be a surface which is obtained by removing $m$ disks from a connected sum of $g$ tori, as shown in Figure~\ref{gen-mcg}~(a).
We call $\Sigma_{g,m}$ a genus $g$ orientable surface with $m$ boundary components.
We define the \textit{mapping class group} $\M(\Sigma_{g,m})$ of $\Sigma_{g,m}$ as the group consisting of isotopy classes of all orientation preserving diffeomorphisms of $\Sigma_{g,m}$.
The \textit{pure mapping class group} $\PM(\Sigma_{g,m})$ of $\Sigma_{g,m}$ is the subgroup of $\M(\Sigma_{g,m})$ consisting of elements which do not permute order of the boundary components of $\Sigma_{g,m}$.
Regarding some boundary component of $\Sigma_{g,n+1}$ as $\ast$, we notice that $\PM(\Sigma_{g,n+1})$ acts on $\pi_1(\Sigma_{g,n},\ast)$ naturally.

For $g\geq1$ and $m\geq0$, let $N_{g,m}$ be a surface which is obtained by removing $m$ disks from a connected sum of $g$ real projective planes.
We call $N_{g,m}$ a genus $g$ non-orientable surface with $m$ boundary components.
We can regard $N_{g,m}$ as a surface which is obtained by attaching $g$ M\"obius bands to $g$ boundary components of $\Sigma_{0,g+m}$, as shown in Figure~\ref{gen-mcg}~(b) or (c).
We call these attached M\"obius bands \textit{crosscaps}.
We define the \textit{mapping class group} $\M(N_{g,m})$ of $N_{g,m}$ as the group consisting of isotopy classes of all diffeomorphisms of $N_{g,m}$.
The \textit{pure mapping class group} $\PM(N_{g,m})$ of $N_{g,m}$ is the subgroup of $\M(N_{g,m})$ consisting of elements which do not permute order of the boundary components of $N_{g,m}$.
Regarding some boundary component of $N_{g,n+1}$ as $\ast$, we notice that $\PM(N_{g,n+1})$ acts on $\pi_1(N_{g,n},\ast)$, and also $\pi_1^+(N_{g,n},\ast)$, naturally.

It is well known that $\PM(\Sigma_{g,m})$ can be generated by only \textit{Dehn twists} (for instance see \cite{D1,D2,L2}).
On the other hand, $\PM(N_{g,m})$ can not be generated by only Dehn twists.
We need \textit{boundary pushing maps} and \textit{crosscap pushing maps} as generators of $\PM(N_{g,m})$, in addition to Dehn twists (see \cite{L1,L3}).
We now define the Dehn twist, the boundary pushing map and the crosscap pushing map.
For a two-sided simple closed curve $c$ of a surface $S$, the Dehn twist $t_c$ about $c$ is the isotopy class of a map as shown in Figure~\ref{DBY}~(a).
When $S$ is orientable, the direction of $t_c$ is the right side with respect to an orientation of $S$.
When $S$ is non-orientable, the direction of $t_c$ is indicated by an arrow written beside $c$ as shown in Figure~\ref{DBY}~(a).
Let $\alpha$ be an oriented arc of $S$ with its two endpoints at a boundary component, as shown in Figure~\ref{DBY}~(b).
The boundary pushing map $B_\alpha$ about $\alpha$ is the isotopy class of a map obtained by pushing the boundary component along $\alpha$.
Let $\alpha$ and $\mu$ be an oriented simple closed curve and a simple closed curve whose regular neighborhood is a crosscap, called a \textit{one-sided simple loop}, of a non-orientable surface, respectively, such that $\alpha$ and $\mu$ intersect transversally at one point, as shown in Figure~\ref{DBY}~(c).
The crosscap pushing map $Y_{\mu,\alpha}$ about $\alpha$ and $\mu$ is the isotopy class of a map obtained by pushing the crosscap, which is the regular neighborhood of $\mu$, along $\alpha$.

%%%%%%%%%%%%%%%%%%%%%%%%%%%%%%%%%%%%%%%%%%%%%%%%%%%%%%%%%%%%%%%%%%%%%%%%%%%%%%%%%%%%%%%%%%%%%%%%%%%%
\begin{figure}[htbp]
\subfigure[The Dehn twist $t_c$ about $c$.]{\includegraphics[scale=0.5]{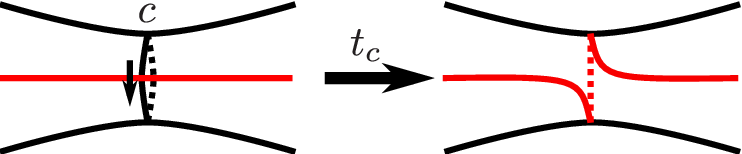}}\\
\subfigure[The boundary pushing map $B_\alpha$ about $\alpha$.]{\includegraphics[scale=0.5]{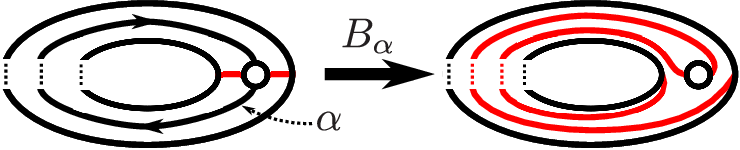}}
\subfigure[The crosscap pushing map $Y_{\mu,\alpha}$ about $\alpha$ and $\mu$.]{\includegraphics[scale=0.5]{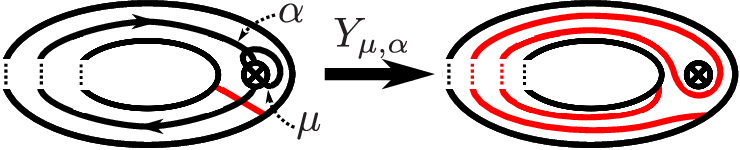}}
\caption{Elements of mapping class groups of surfaces.}\label{DBY}
\end{figure}
%%%%%%%%%%%%%%%%%%%%%%%%%%%%%%%%%%%%%%%%%%%%%%%%%%%%%%%%%%%%%%%%%%%%%%%%%%%%%%%%%%%%%%%%%%%%%%%%%%%%

We have the following theorems.

%%%%%%%%%%%%%%%%%%%%%%%%%%%%%%%%%%%%%%%%%%%%%%%%%%%%%%%%%%%%%%%%%%%%%%%%%%%%%%%%%%%%%%%%%%%%%%%%%%%%
%%%%%%%%%%%%%%%%%%%%%%%%%%%%%%%%%%%%%%%%%%%%%%%%%%%%%%%%%%%%%%%%%%%%%%%%%%%%%%%%%%%%%%%%%%%%%%%%%%%%
\begin{thm}[c.f. \cite{FM}]\label{gen-PMS}
Let $c_0$, $c_1,\dots,c_{2g}$ and $d_1,\dots,d_n$ be simple closed curves of $\Sigma_{g,n+1}$ as shown in Figure~\ref{gen-mcg}~(a).
Then $\PM(\Sigma_{g,n+1})$ is generated by $t_{c_0}$, $t_{c_1},\dots,t_{c_{2g}}$ and $t_{d_1},\dots,t_{d_n}$.
\end{thm}
%%%%%%%%%%%%%%%%%%%%%%%%%%%%%%%%%%%%%%%%%%%%%%%%%%%%%%%%%%%%%%%%%%%%%%%%%%%%%%%%%%%%%%%%%%%%%%%%%%%%
%%%%%%%%%%%%%%%%%%%%%%%%%%%%%%%%%%%%%%%%%%%%%%%%%%%%%%%%%%%%%%%%%%%%%%%%%%%%%%%%%%%%%%%%%%%%%%%%%%%%

%%%%%%%%%%%%%%%%%%%%%%%%%%%%%%%%%%%%%%%%%%%%%%%%%%%%%%%%%%%%%%%%%%%%%%%%%%%%%%%%%%%%%%%%%%%%%%%%%%%%
%%%%%%%%%%%%%%%%%%%%%%%%%%%%%%%%%%%%%%%%%%%%%%%%%%%%%%%%%%%%%%%%%%%%%%%%%%%%%%%%%%%%%%%%%%%%%%%%%%%%
\begin{thm}\label{gen-PMF}
Let $a_1,\dots,a_{g-1}$, $b$, $\mu$, $s_{kl}$ and $r_0$, $r_1,\cdots,r_n$ be simple closed curves and simple arcs of $N_{g,n+1}$ for $1\leq{k<l}\leq{}n$, as shown in Figures~\ref{gen-mcg}~(b) and (c).
Then $\PM(N_{g,n+1})$ is generated by $t_{a_1},\dots,t_{a_{g-1}}$, $t_b$, $Y_{\mu,a_1}$, $t_{s_{kl}}$ and $B_{r_0}$, $B_{r_1},\dots,B_{r_n}$ for $1\leq{k<l}\leq{}n$.
\end{thm}
%%%%%%%%%%%%%%%%%%%%%%%%%%%%%%%%%%%%%%%%%%%%%%%%%%%%%%%%%%%%%%%%%%%%%%%%%%%%%%%%%%%%%%%%%%%%%%%%%%%%
%%%%%%%%%%%%%%%%%%%%%%%%%%%%%%%%%%%%%%%%%%%%%%%%%%%%%%%%%%%%%%%%%%%%%%%%%%%%%%%%%%%%%%%%%%%%%%%%%%%%

%%%%%%%%%%%%%%%%%%%%%%%%%%%%%%%%%%%%%%%%%%%%%%%%%%%%%%%%%%%%%%%%%%%%%%%%%%%%%%%%%%%%%%%%%%%%%%%%%%%%
\begin{figure}[htbp]
\subfigure[Simple closed curves $c_i$ and $d_k$ of $\Sigma_{g,n+1}$ for $0\leq{i}\leq2g$ and $1\leq{k}\leq{n}$.]{\includegraphics[scale=0.5]{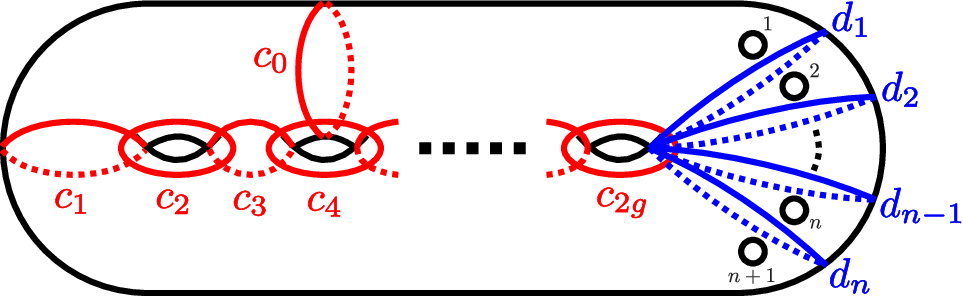}}\\
\subfigure[Simple closed curves and simple arcs $a_i$, $b$, $\mu$ and $r_k$ of $N_{g,n+1}$ for $1\leq{i}\leq{g-1}$ and $0\leq{k}\leq{n}$.]{\includegraphics[scale=0.5]{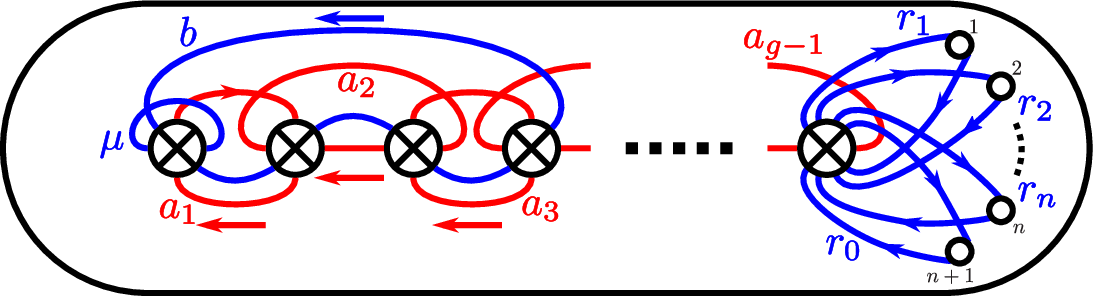}}
\subfigure[A simple closed curve $s_{kl}$ of $N_{g,n+1}$ for $1\leq{k<l}\leq{}n$.]{\includegraphics[scale=0.5]{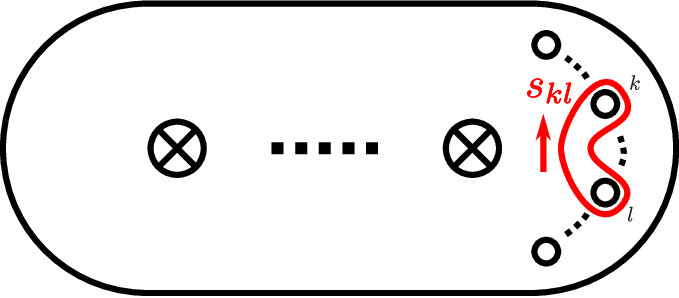}}
\caption{}\label{gen-mcg}
\end{figure}
%%%%%%%%%%%%%%%%%%%%%%%%%%%%%%%%%%%%%%%%%%%%%%%%%%%%%%%%%%%%%%%%%%%%%%%%%%%%%%%%%%%%%%%%%%%%%%%%%%%%

%%%%%%%%%%%%%%%%%%%%%%%%%%%%%%%%%%%%%%%%%%%%%%%%%%%%%%%%%%%%%%%%%%%%%%%%%%%%%%%%%%%%%%%%%%%%%%%%%%%%
%%%%%%%%%%%%%%%%%%%%%%%%%%%%%%%%%%%%%%%%%%%%%%%%%%%%%%%%%%%%%%%%%%%%%%%%%%%%%%%%%%%%%%%%%%%%%%%%%%%%
\begin{proof}
There is an exact sequence
$$\pi_1(N_{g,n},\ast)\to\PM(N_{g,n+1})\to\PM(N_{g,n})\to1,$$
introduced by Birman~\cite{B} for orientable surfaces.
The homomorphism $\pi_1(N_{g,n},\ast)\to\PM(N_{g,n+1})$ is defined as $\alpha\mapsto{}B_{\overline{\alpha}}$, where $\overline{\alpha}$ is an arc which is obtained from $\alpha$ by regarding $\ast$ as a boundary component.
The homomorphism $\PM(N_{g,n+1})\to\PM(N_{g,n})$ is defined as the map which is induced by capping the boundary component with a disk.

Let $x_1,\dots,x_g$ and $y_1,\dots,y_{n-1}$ be oriented simple loops of $N_{g,n}$ based at $\ast$, as shown in Figure~\ref{gen-pi_1-non-ori-surf}.
It is well known that $\pi_1(N_{g,n},\ast)$ is generated by $x_1,\dots,x_g$ and $y_1,\dots,y_{n-1}$.
It is easy to check that $t_{a_i}F^{i-1}Y_{\mu,a_1}^{(-1)^{i-1}}F^{1-i}(x_{i+1})=x_i$ for $i=1,\dots,g-1$, where $F=t_{a_1}t_{a_2}\cdots{}t_{a_{g-1}}$.
Therefore, since the homomorphism $\pi_1(N_{g,0},\ast)\to\PM(N_{g,1})$ sends $x_g$ to $B_{r_0}$, we see that this homomorphism sends  $x_i$ to a conjugate element of $B_{r_0}$ by $t_{a_1},\dots,t_{a_{g-1}}$ and $Y_{\mu,a_1}$ from the relation $B_{f(r_0)}=fB_{r_0}f^{-1}$ (for example see Lemma~2.4 in \cite{Kor}).
In addition, regarding $\ast$ as the $n$-th boundary component of $N_{g,n+1}$, it follows that the homomorphism $\pi_1(N_{g,n},\ast)\to\PM(N_{g,n+1})$ sends $x_i$ to a conjugate element of $B_{r_n}$ by $t_{a_1},\dots,t_{a_{g-1}}$ and $Y_{\mu,a_1}$ for $n\geq1$ and $1\leq{i}\leq{g}$ from a similar argument, and $y_k$ to $s_{kn}$ for $n\geq2$ and $1\leq{k}\leq{n-1}$.
It is known that $\PM(N_{g,0})$ is generated by $t_{a_1},\dots,t_{a_{g-1}}$, $t_b$ and $Y_{\mu,a_1}$ (see \cite{C,S2}).
Therefore using the exact sequence above, we obtain the generating set inductively.

%%%%%%%%%%%%%%%%%%%%%%%%%%%%%%%%%%%%%%%%%%%%%%%%%%%%%%%%%%%%%%%%%%%%%%%%%%%%%%%%%%%%%%%%%%%%%%%%%%%%
\begin{figure}[htbp]
\includegraphics[scale=0.5]{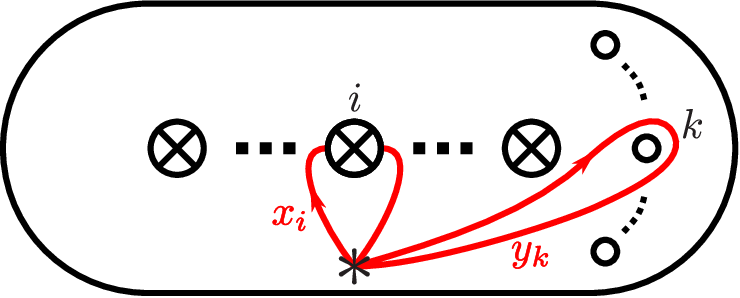}
\caption{Oriented simple loops $x_i$ and $y_k$ of $N_{g,n}$ based at $\ast$ for $1\leq{i}\leq{g}$ and $1\leq{k}\leq{n}$.}\label{gen-pi_1-non-ori-surf}
\end{figure}
%%%%%%%%%%%%%%%%%%%%%%%%%%%%%%%%%%%%%%%%%%%%%%%%%%%%%%%%%%%%%%%%%%%%%%%%%%%%%%%%%%%%%%%%%%%%%%%%%%%%
\end{proof}
%%%%%%%%%%%%%%%%%%%%%%%%%%%%%%%%%%%%%%%%%%%%%%%%%%%%%%%%%%%%%%%%%%%%%%%%%%%%%%%%%%%%%%%%%%%%%%%%%%%%
%%%%%%%%%%%%%%%%%%%%%%%%%%%%%%%%%%%%%%%%%%%%%%%%%%%%%%%%%%%%%%%%%%%%%%%%%%%%%%%%%%%%%%%%%%%%%%%%%%%%

Note that a finite generating set of $\PM(N_{g,n+1})$ which is different from that of Theorem~\ref{gen-PMF} was already given by Korkmaz~\cite{Kor}.
However we use the generating set of Theorem~\ref{gen-PMF} in this paper.

%%%%%%%%%%%%%%%%%%%%%%%%%%%%%%%%%%%%%%%%%%%%%%%%%%%%%%%%%%%%%%%%%%%%%%%%%%%%%%%%%%%%%%%%%%%%%%%%%%%%
%%%%%%%%%%%%%%%%%%%%%%%%%%%%%%%%%%%%%%%%%%%%%%%%%%%%%%%%%%%%%%%%%%%%%%%%%%%%%%%%%%%%%%%%%%%%%%%%%%%%
%%%%%%%%%%%%%%%%%%%%%%%%%%%%%%%%%%%%%%%%%%%%%%%%%%%%%%%%%%%%%%%%%%%%%%%%%%%%%%%%%%%%%%%%%%%%%%%%%%%%
%%%%%%%%%%%%%%%%%%%%%%%%%%%%%%%%%%%%%%%%%%%%%%%%%%%%%%%%%%%%%%%%%%%%%%%%%%%%%%%%%%%%%%%%%%%%%%%%%%%%
%%%%%%%%%%%%%%%%%%%%%%%%%%%%%%%%%%%%%%%%%%%%%%%%%%%%%%%%%%%%%%%%%%%%%%%%%%%%%%%%%%%%%%%%%%%%%%%%%%%%
%%%%%%%%%%%%%%%%%%%%%%%%%%%%%%%%%%%%%%%%%%%%%%%%%%%%%%%%%%%%%%%%%%%%%%%%%%%%%%%%%%%%%%%%%%%%%%%%%%%%
%%%%%%%%%%%%%%%%%%%%%%%%%%%%%%%%%%%%%%%%%%%%%%%%%%%%%%%%%%%%%%%%%%%%%%%%%%%%%%%%%%%%%%%%%%%%%%%%%%%%
%%%%%%%%%%%%%%%%%%%%%%%%%%%%%%%%%%%%%%%%%%%%%%%%%%%%%%%%%%%%%%%%%%%%%%%%%%%%%%%%%%%%%%%%%%%%%%%%%%%%
%%%%%%%%%%%%%%%%%%%%%%%%%%%%%%%%%%%%%%%%%%%%%%%%%%%%%%%%%%%%%%%%%%%%%%%%%%%%%%%%%%%%%%%%%%%%%%%%%%%%
%%%%%%%%%%%%%%%%%%%%%%%%%%%%%%%%%%%%%%%%%%%%%%%%%%%%%%%%%%%%%%%%%%%%%%%%%%%%%%%%%%%%%%%%%%%%%%%%%%%%
\section{Proof of Theorem~\ref{main-1}}\label{pi}

In Subsection~\ref{ori}, we prove Theorem~\ref{main-1} of the case where $S$ is orientable.
In Subsection~\ref{non-ori}, we prove Theorem~\ref{main-1} of the case where $S$ is non-orientable.

%%%%%%%%%%%%%%%%%%%%%%%%%%%%%%%%%%%%%%%%%%%%%%%%%%%%%%%%%%%%%%%%%%%%%%%%%%%%%%%%%%%%%%%%%%%%%%%%%%%%
%%%%%%%%%%%%%%%%%%%%%%%%%%%%%%%%%%%%%%%%%%%%%%%%%%%%%%%%%%%%%%%%%%%%%%%%%%%%%%%%%%%%%%%%%%%%%%%%%%%%
%%%%%%%%%%%%%%%%%%%%%%%%%%%%%%%%%%%%%%%%%%%%%%%%%%%%%%%%%%%%%%%%%%%%%%%%%%%%%%%%%%%%%%%%%%%%%%%%%%%%
%%%%%%%%%%%%%%%%%%%%%%%%%%%%%%%%%%%%%%%%%%%%%%%%%%%%%%%%%%%%%%%%%%%%%%%%%%%%%%%%%%%%%%%%%%%%%%%%%%%%
%%%%%%%%%%%%%%%%%%%%%%%%%%%%%%%%%%%%%%%%%%%%%%%%%%%%%%%%%%%%%%%%%%%%%%%%%%%%%%%%%%%%%%%%%%%%%%%%%%%%
\subsection{The case where $S$ is orientable}\label{ori}\

Let $\alpha_1,\dots,\alpha_g$, $\beta_1,\dots,\beta_g$ and $\gamma_1,\dots,\gamma_{n-1}$ be oriented simple loops of $\Sigma_{g,n}$ based at $\ast$, as shown in Figure~\ref{gen-pi_1-ori-surf}.
It is well known that $\pi_1(\Sigma_{g,n},\ast)$ is the free group freely generated by these loops for $n\geq1$ and the group generated by $\alpha_1,\dots,\alpha_g$ and $\beta_1,\dots,\beta_g$ which has one relation $[\alpha_1,\beta_1]\cdots[\alpha_g,\beta_g]=1$ for $n=0$, where $[x,y]=xyx^{-1}y^{-1}$.

%%%%%%%%%%%%%%%%%%%%%%%%%%%%%%%%%%%%%%%%%%%%%%%%%%%%%%%%%%%%%%%%%%%%%%%%%%%%%%%%%%%%%%%%%%%%%%%%%%%%
\begin{figure}[htbp]
\includegraphics[scale=0.5]{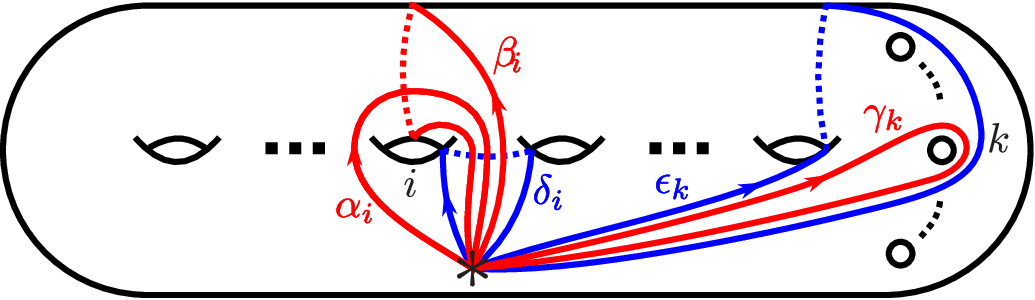}
\caption{Oriented simple loops $\alpha_i$, $\beta_i$, $\gamma_k$, $\delta_i$ and $\epsilon_k$ of $\Sigma_{g,n}$ based at $\ast$ for $1\leq{i}\leq{g}$ and $1\leq{k}\leq{n}$, except for $i=g$ for $\delta_i$.}\label{gen-pi_1-ori-surf}
\end{figure}
%%%%%%%%%%%%%%%%%%%%%%%%%%%%%%%%%%%%%%%%%%%%%%%%%%%%%%%%%%%%%%%%%%%%%%%%%%%%%%%%%%%%%%%%%%%%%%%%%%%%

Let $X$ be a set consisting of $S_\alpha$, where $\alpha$ is a non-separating simple loop or a separating simple loop which bounds the $m$-th boundary component for $1\leq{m}\leq{n-1}$, and let $X^\prime$ be the following subset of $X$:
$$X^\prime=\{S_{\alpha_1},\dots,S_{\alpha_g},S_{\beta_1},\dots,S_{\beta_g},S_{\gamma_1},\dots,S_{\gamma_{n-1}}\}.$$
Let $Y$ be the generating set for $\PM(\Sigma_{g,n+1})$ given in Theorem~\ref{gen-PMS}.
In the actions on $\pi_1(\Sigma_{g,n},\ast)$ and $\pi$ by $\PM(\Sigma_{g,n+1})$, we regard the $(n+1)$-st boundary component of $\Sigma_{g,n+1}$ as $\ast$.
We define $f(S_\alpha)=S_{f_\sharp(\alpha)}$ for $S_\alpha\in\pi$ and $f\in\PM(\Sigma_{g,n+1})$, where $f_\sharp$ is the map on $\pi_1(\Sigma_{g,n},\ast)$ induced from $f$.
We prove the following proposition.

%%%%%%%%%%%%%%%%%%%%%%%%%%%%%%%%%%%%%%%%%%%%%%%%%%%%%%%%%%%%%%%%%%%%%%%%%%%%%%%%%%%%%%%%%%%%%%%%%%%%
%%%%%%%%%%%%%%%%%%%%%%%%%%%%%%%%%%%%%%%%%%%%%%%%%%%%%%%%%%%%%%%%%%%%%%%%%%%%%%%%%%%%%%%%%%%%%%%%%%%%
\begin{prop}\label{1}
\begin{enumerate}
\item	$X$ generates $\pi$.
\item	$\PM(\Sigma_{g,n+1})(X^\prime)=X$.
\item	For any $x\in{X^\prime}$ and $y\in{Y}$, $y^{\pm1}(x)$ is in the subgroup of $\pi$ generated by $X^\prime$.
\end{enumerate}
\end{prop}
%%%%%%%%%%%%%%%%%%%%%%%%%%%%%%%%%%%%%%%%%%%%%%%%%%%%%%%%%%%%%%%%%%%%%%%%%%%%%%%%%%%%%%%%%%%%%%%%%%%%
%%%%%%%%%%%%%%%%%%%%%%%%%%%%%%%%%%%%%%%%%%%%%%%%%%%%%%%%%%%%%%%%%%%%%%%%%%%%%%%%%%%%%%%%%%%%%%%%%%%%

In order to prove the proposition, we show the following lemma.

%%%%%%%%%%%%%%%%%%%%%%%%%%%%%%%%%%%%%%%%%%%%%%%%%%%%%%%%%%%%%%%%%%%%%%%%%%%%%%%%%%%%%%%%%%%%%%%%%%%%
%%%%%%%%%%%%%%%%%%%%%%%%%%%%%%%%%%%%%%%%%%%%%%%%%%%%%%%%%%%%%%%%%%%%%%%%%%%%%%%%%%%%%%%%%%%%%%%%%%%%
\begin{lem}\label{d-e}
For $1\leq{i}\leq{g-1}$ and $1\leq{k}\leq{n}$, $S_{\gamma_n}$, $S_{\delta_i}$ and $S_{\epsilon_k}$ are in the subgroup of $\pi$ generated by $X^\prime$, where $\gamma_n$, $\delta_i$ and $\epsilon_k$ are simple loops of $\Sigma_{g,n}$ based at $\ast$ as shown in Figure~\ref{gen-pi_1-ori-surf}.
\end{lem}
%%%%%%%%%%%%%%%%%%%%%%%%%%%%%%%%%%%%%%%%%%%%%%%%%%%%%%%%%%%%%%%%%%%%%%%%%%%%%%%%%%%%%%%%%%%%%%%%%%%%
%%%%%%%%%%%%%%%%%%%%%%%%%%%%%%%%%%%%%%%%%%%%%%%%%%%%%%%%%%%%%%%%%%%%%%%%%%%%%%%%%%%%%%%%%%%%%%%%%%%%

%%%%%%%%%%%%%%%%%%%%%%%%%%%%%%%%%%%%%%%%%%%%%%%%%%%%%%%%%%%%%%%%%%%%%%%%%%%%%%%%%%%%%%%%%%%%%%%%%%%%
%%%%%%%%%%%%%%%%%%%%%%%%%%%%%%%%%%%%%%%%%%%%%%%%%%%%%%%%%%%%%%%%%%%%%%%%%%%%%%%%%%%%%%%%%%%%%%%%%%%%
\begin{proof}
By the relations (1) and (2) of $\pi$, we calculate
\begin{eqnarray*}
S_{\gamma_n}&=&S_{([\alpha_1,\beta_1]\cdots[\alpha_g,\beta_g]\gamma_1\cdots\gamma_{n-1})^{-1}}
=([S_{\alpha_1},S_{\beta_1}]\cdots[S_{\alpha_g},S_{\beta_g}]S_{\gamma_1}\cdots{}S_{\gamma_{n-1}})^{-1},\\
S_{\delta_i}&=&S_{\beta_i^{-1}\alpha_{i+1}\beta_{i+1}\alpha_{i+1}^{-1}}
=S_{\beta_i}^{-1}S_{\alpha_{i+1}}S_{\beta_{i+1}}S_{\alpha_{i+1}}^{-1},\\
S_{\epsilon_k}&=&S_{\beta_g^{-1}\gamma_1\cdots\gamma_k}
=S_{\beta_g}^{-1}S_{\gamma_1}\cdots{}S_{\gamma_k}
\end{eqnarray*}
for $1\leq{i}\leq{g-1}$ and $1\leq{k}\leq{n}$.
Since each symbol of the right hand sides is in $X^\prime$, we get the claim.
\end{proof}
%%%%%%%%%%%%%%%%%%%%%%%%%%%%%%%%%%%%%%%%%%%%%%%%%%%%%%%%%%%%%%%%%%%%%%%%%%%%%%%%%%%%%%%%%%%%%%%%%%%%
%%%%%%%%%%%%%%%%%%%%%%%%%%%%%%%%%%%%%%%%%%%%%%%%%%%%%%%%%%%%%%%%%%%%%%%%%%%%%%%%%%%%%%%%%%%%%%%%%%%%

%%%%%%%%%%%%%%%%%%%%%%%%%%%%%%%%%%%%%%%%%%%%%%%%%%%%%%%%%%%%%%%%%%%%%%%%%%%%%%%%%%%%%%%%%%%%%%%%%%%%
%%%%%%%%%%%%%%%%%%%%%%%%%%%%%%%%%%%%%%%%%%%%%%%%%%%%%%%%%%%%%%%%%%%%%%%%%%%%%%%%%%%%%%%%%%%%%%%%%%%%
\begin{proof}[Proof of Proposition~\ref{1}]
(1)	For any generator $S_\alpha$ of $\pi$, if $\alpha$ is a non-separating simple loop, $S_\alpha$ is in $X$ clearly.
	If $\alpha$ is a separating simple loop, one of a component of the complement of $\alpha$ is homeomorphic to $\Sigma_{h,m+1}$ for some $0\leq{h}\leq{g}$ and $0\leq{m}\leq{n}$.
	Therefore, there is $f\in\PM(\Sigma_{g,n+1})$ such that $\alpha=f_\sharp([\alpha_1,\beta_1]\cdots[\alpha_h,\beta_h]\gamma_{k_1}\cdots\gamma_{k_m})$ for some $1\leq{k_1<\cdots<k_m}\leq{n}$ (see Figure~\ref{normal-position-loop-ori-surf}).
	Then, by the relation (2) of $\pi$, we have $S_\alpha=[S_{f_\sharp(\alpha_1)},S_{f_\sharp(\beta_1)}]\cdots[S_{f_\sharp(\alpha_h)},S_{f_\sharp(\beta_h)}]S_{f_\sharp(\gamma_{k_1})}\cdots{}S_{f_\sharp(\gamma_{k_m})}$.
	Since each symbol of the right hand side is in $X$, we conclude that $X$ generates $\pi$.
	%%%%%%%%%%%%%%%%%%%%%%%%%%%%%%%%%%%%%%%%%%%%%%%%%%%%%%%%%%%%%%%%%%%%%%%%%%%%%%%%%%%%%%%%%%%%%%%%%%%%
	\begin{figure}[htbp]
	\includegraphics[scale=0.5]{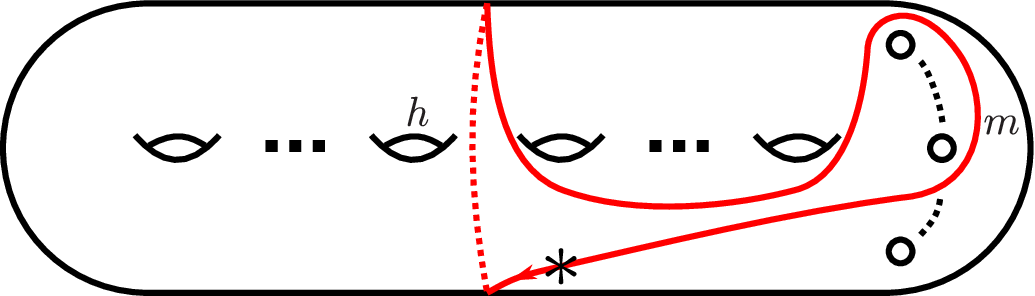}
	\caption{An oriented simple loop $[\alpha_1,\beta_1]\cdots[\alpha_h,\beta_h]\gamma_1\cdots\gamma_m$, one of a component of whose complement is homeomorphic to $\Sigma_{h,m+1}$, for $0\leq{h}\leq{g}$ and $0\leq{m}\leq{n}$.}\label{normal-position-loop-ori-surf}
	\end{figure}
	%%%%%%%%%%%%%%%%%%%%%%%%%%%%%%%%%%%%%%%%%%%%%%%%%%%%%%%%%%%%%%%%%%%%%%%%%%%%%%%%%%%%%%%%%%%%%%%%%%%%

(2)	For any $S_\alpha\in{X}$, if $\alpha$ is a non-separating simple loop, there is $f\in\PM(\Sigma_{g,n+1})$ such that $f_\sharp(\alpha_1)=\alpha$, and hence $f(S_{\alpha_1})=S_\alpha$.
	If $\alpha$ is a separating simple loop which bounds the $m$-th boundary component for $1\leq{m}\leq{n-1}$, there is $f\in\PM(\Sigma_{g,n+1})$ such that $f_\sharp(\gamma_m)=\alpha$, and hence $f(S_{\gamma_m})=S_\alpha$.
	Therefore we obtain the claim.

(3)	In this proof, we omit details of calculations.
	
	Let $y=t_{c_0}$.
	We calculate
	$$
	y(S_{\alpha_2})=S_{\alpha_2\beta_2^{-1}}\overset{(1),(2)}{=}S_{\alpha_2}S_{\beta_2}^{-1},~
	y^{-1}(S_{\alpha_2})=S_{\alpha_2\beta_2}\overset{(2)}{=}S_{\alpha_2}S_{\beta_2}
	$$
	and $y^{\pm1}(x)=x$ for any other $x\in{X^\prime}$.
	
	Let $y=t_{c_{2i-1}}$ for $1\leq{i}\leq{g}$.
	We calculate
	\begin{eqnarray*}
	y(S_{\alpha_{i-1}})&=&S_{\alpha_{i-1}\delta_{i-1}}\overset{(2)}{=}S_{\alpha_{i-1}}S_{\delta_{i-1}},\\
	y^{-1}(S_{\alpha_{i-1}})&=&S_{\alpha_{i-1}\delta_{i-1}^{-1}}\overset{(1),(2)}{=}S_{\alpha_{i-1}}S_{\delta_{i-1}}^{-1},\\
	y(S_{\alpha_i})&=&S_{\delta_{i-1}^{-1}\alpha_i}\overset{(1),(2)}{=}S_{\delta_{i-1}}^{-1}S_{\alpha_i},~
	y^{-1}(S_{\alpha_i})=S_{\delta_{i-1}\alpha_i}\overset{(2)}{=}S_{\delta_{i-1}}S_{\alpha_i},\\
	y^{\pm1}(S_{\beta_{i-1}})&=&S_{\delta_{i-1}^{\mp1}\beta_{i-1}\delta_{i-1}^{\pm1}}\overset{(1),(2)}{=}S_{\delta_{i-1}}^{\mp1}S_{\beta_{i-1}}S_{\delta_{i-1}}^{\pm1}\\
	\end{eqnarray*}
	and $y^{\pm1}(x)=x$ for any other $x\in{X^\prime}$.
	
	Let $y=t_{c_{2i}}$ for $1\leq{i}\leq{g}$.
	We calculate
	$$
	y(S_{\beta_i})=S_{\beta_i\alpha_i}\overset{(2)}{=}S_{\beta_i}S_{\alpha_i},
	y^{-1}(S_{\beta_i})=S_{\beta_i\alpha_i^{-1}}\overset{(1),(2)}{=}S_{\beta_i}S_{\alpha_i}^{-1}
	$$
	and $y^{\pm1}(x)=x$ for any other $x\in{X^\prime}$.

	Let $y=t_{d_k}$ for $1\leq{k}\leq{n}$.
	We calculate
	\begin{eqnarray*}
	y(S_{\alpha_g})&=&S_{\alpha_g\epsilon_k}\overset{(2)}{=}S_{\alpha_g}S_{\epsilon_k},~
	y^{-1}(S_{\alpha_g})=S_{\alpha_g\epsilon_k^{-1}}\overset{(1),(2)}{=}S_{\alpha_g}S_{\epsilon_k}^{-1},\\
	y^{\pm1}(S_{\beta_g})&=&S_{\epsilon_k^{\mp1}\beta_g\epsilon_k^{\pm1}}\overset{(1),(2)}{=}S_{\epsilon_k}^{\mp1}S_{\beta_g}S_{\epsilon_k}^{\pm1},\\
	y^{\pm1}(S_{\gamma_l})&=&S_{\epsilon_k^{\mp1}\gamma_l\epsilon_k^{\pm1}}\overset{(1),(2)}{=}S_{\epsilon_k}^{\mp1}S_{\gamma_l}S_{\epsilon_k}^{\pm1}
	\end{eqnarray*}
	for $l\leq{k}$, and $y^{\pm1}(x)=x$ for any other $x\in{X^\prime}$.
	
	Hence we have that for any $x\in{X^\prime}$ and $y\in{Y}$, $y^{\pm1}(x)$ is in the subgroup of $\pi$ generated by $X^\prime$, by Lemma~\ref{d-e}.
\end{proof}
%%%%%%%%%%%%%%%%%%%%%%%%%%%%%%%%%%%%%%%%%%%%%%%%%%%%%%%%%%%%%%%%%%%%%%%%%%%%%%%%%%%%%%%%%%%%%%%%%%%%
%%%%%%%%%%%%%%%%%%%%%%%%%%%%%%%%%%%%%%%%%%%%%%%%%%%%%%%%%%%%%%%%%%%%%%%%%%%%%%%%%%%%%%%%%%%%%%%%%%%%

%%%%%%%%%%%%%%%%%%%%%%%%%%%%%%%%%%%%%%%%%%%%%%%%%%%%%%%%%%%%%%%%%%%%%%%%%%%%%%%%%%%%%%%%%%%%%%%%%%%%
%%%%%%%%%%%%%%%%%%%%%%%%%%%%%%%%%%%%%%%%%%%%%%%%%%%%%%%%%%%%%%%%%%%%%%%%%%%%%%%%%%%%%%%%%%%%%%%%%%%%
\begin{proof}[Proof of Theorem~\ref{main-1} of the case where $S$ is orientable]
By Lemma~\ref{main-lem} and Proposition~\ref{1}, it follows that $\pi$ is generated by $X^\prime$.
There is a natural map $\pi\to\pi_1(\Sigma_{g,n},\ast)$.
The relations~(1) and (2) of $\pi$ are satisfied in $\pi_1(\Sigma_{g,n},\ast)$ clearly.
Hence the map is a homomorphism.
In addition, the relation $[S_{\alpha_1},S_{\beta_1}]\cdots[S_{\alpha_g},S_{\beta_g}]=1$ is obtained from the relation~(2) of $\pi$ for $n=0$.
Therefore the map is an isomorphism for any $n\geq0$.
Thus we complete the proof.
\end{proof}
%%%%%%%%%%%%%%%%%%%%%%%%%%%%%%%%%%%%%%%%%%%%%%%%%%%%%%%%%%%%%%%%%%%%%%%%%%%%%%%%%%%%%%%%%%%%%%%%%%%%
%%%%%%%%%%%%%%%%%%%%%%%%%%%%%%%%%%%%%%%%%%%%%%%%%%%%%%%%%%%%%%%%%%%%%%%%%%%%%%%%%%%%%%%%%%%%%%%%%%%%

%%%%%%%%%%%%%%%%%%%%%%%%%%%%%%%%%%%%%%%%%%%%%%%%%%%%%%%%%%%%%%%%%%%%%%%%%%%%%%%%%%%%%%%%%%%%%%%%%%%%
%%%%%%%%%%%%%%%%%%%%%%%%%%%%%%%%%%%%%%%%%%%%%%%%%%%%%%%%%%%%%%%%%%%%%%%%%%%%%%%%%%%%%%%%%%%%%%%%%%%%
%%%%%%%%%%%%%%%%%%%%%%%%%%%%%%%%%%%%%%%%%%%%%%%%%%%%%%%%%%%%%%%%%%%%%%%%%%%%%%%%%%%%%%%%%%%%%%%%%%%%
%%%%%%%%%%%%%%%%%%%%%%%%%%%%%%%%%%%%%%%%%%%%%%%%%%%%%%%%%%%%%%%%%%%%%%%%%%%%%%%%%%%%%%%%%%%%%%%%%%%%
%%%%%%%%%%%%%%%%%%%%%%%%%%%%%%%%%%%%%%%%%%%%%%%%%%%%%%%%%%%%%%%%%%%%%%%%%%%%%%%%%%%%%%%%%%%%%%%%%%%%
\subsection{The case where $S$ is non-orientable}\label{non-ori}\

Let $x_1,\dots,x_g$ and $y_1,\dots,y_{n-1}$ be oriented simple loops of $N_{g,n}$ based at $\ast$, as shown in Figure~\ref{gen-pi_1-non-ori-surf}.
It is well known that $\pi_1(N_{g,n},\ast)$ is the free group freely generated by these loops for $n\geq1$ and the group generated by $x_1,\dots,x_g$ which has one relation $x_1^2\cdots{}x_g^2=1$ for $n=0$.

Let $X$ be a set consisting of $S_\alpha$, where $\alpha$ is a one-sided simple loop whose complement is non-orientable, or a separating simple loop which bounds the $m$-th boundary component for $1\leq{m}\leq{n-1}$, and let $X^\prime$ be the following subset of $X$:
$$X^\prime=\{S_{x_1},\dots,S_{x_g},S_{y_1},\dots,S_{y_{n-1}}\}.$$
Let $Y$ be the generating set for $\PM(N_{g,n+1})$ given in Theorem~\ref{gen-PMF}.
In the actions on $\pi_1(N_{g,n},\ast)$ and $\pi$ by $\PM(N_{g,n+1})$, we regard the $(n+1)$-st boundary component of $N_{g,n+1}$ as $\ast$.
We define $f(S_\alpha)=S_{f_\sharp(\alpha)}$ for $S_\alpha\in\pi$ and $f\in\PM(N_{g,n+1})$, where $f_\sharp$ is the map on $\pi_1(N_{g,n},\ast)$ induced from $f$.
We prove the following proposition.

%%%%%%%%%%%%%%%%%%%%%%%%%%%%%%%%%%%%%%%%%%%%%%%%%%%%%%%%%%%%%%%%%%%%%%%%%%%%%%%%%%%%%%%%%%%%%%%%%%%%
%%%%%%%%%%%%%%%%%%%%%%%%%%%%%%%%%%%%%%%%%%%%%%%%%%%%%%%%%%%%%%%%%%%%%%%%%%%%%%%%%%%%%%%%%%%%%%%%%%%%
\begin{prop}\label{2}
\begin{enumerate}
\item	$X$ generates $\pi$.
\item	$\PM(N_{g,n+1})(X^\prime)=X$.
\item	For any $x\in{X^\prime}$ and $y\in{Y}$, $y^{\pm1}(x)$ is in the subgroup of $\pi$ generated by $X^\prime$.
\end{enumerate}
\end{prop}
%%%%%%%%%%%%%%%%%%%%%%%%%%%%%%%%%%%%%%%%%%%%%%%%%%%%%%%%%%%%%%%%%%%%%%%%%%%%%%%%%%%%%%%%%%%%%%%%%%%%
%%%%%%%%%%%%%%%%%%%%%%%%%%%%%%%%%%%%%%%%%%%%%%%%%%%%%%%%%%%%%%%%%%%%%%%%%%%%%%%%%%%%%%%%%%%%%%%%%%%%

In order to prove the proposition, we show the following lemma.

%%%%%%%%%%%%%%%%%%%%%%%%%%%%%%%%%%%%%%%%%%%%%%%%%%%%%%%%%%%%%%%%%%%%%%%%%%%%%%%%%%%%%%%%%%%%%%%%%%%%
%%%%%%%%%%%%%%%%%%%%%%%%%%%%%%%%%%%%%%%%%%%%%%%%%%%%%%%%%%%%%%%%%%%%%%%%%%%%%%%%%%%%%%%%%%%%%%%%%%%%
\begin{lem}\label{y_n}
$S_{y_n}$ is in the subgroup of $\pi$ generated by $X^\prime$, where $y_n$ is a simple loop of $N_{g,n}$ as shown in Figure~\ref{gen-pi_1-non-ori-surf}.
\end{lem}
%%%%%%%%%%%%%%%%%%%%%%%%%%%%%%%%%%%%%%%%%%%%%%%%%%%%%%%%%%%%%%%%%%%%%%%%%%%%%%%%%%%%%%%%%%%%%%%%%%%%
%%%%%%%%%%%%%%%%%%%%%%%%%%%%%%%%%%%%%%%%%%%%%%%%%%%%%%%%%%%%%%%%%%%%%%%%%%%%%%%%%%%%%%%%%%%%%%%%%%%%

%%%%%%%%%%%%%%%%%%%%%%%%%%%%%%%%%%%%%%%%%%%%%%%%%%%%%%%%%%%%%%%%%%%%%%%%%%%%%%%%%%%%%%%%%%%%%%%%%%%%
%%%%%%%%%%%%%%%%%%%%%%%%%%%%%%%%%%%%%%%%%%%%%%%%%%%%%%%%%%%%%%%%%%%%%%%%%%%%%%%%%%%%%%%%%%%%%%%%%%%%
\begin{proof}
By the relations (1) and (2) of $\pi$, we calculate
$$
S_{y_n}=S_{(x_1^2\cdots{}x_g^2y_1\cdots{}y_{n-1})^{-1}}
=(S_{x_1}^2\cdots{}S_{x_g}^2S_{y_1}\cdots{}S_{y_{n-1}})^{-1}.
$$
Since each symbol of the right hand side is in $X^\prime$, we get the claim.
\end{proof}
%%%%%%%%%%%%%%%%%%%%%%%%%%%%%%%%%%%%%%%%%%%%%%%%%%%%%%%%%%%%%%%%%%%%%%%%%%%%%%%%%%%%%%%%%%%%%%%%%%%%
%%%%%%%%%%%%%%%%%%%%%%%%%%%%%%%%%%%%%%%%%%%%%%%%%%%%%%%%%%%%%%%%%%%%%%%%%%%%%%%%%%%%%%%%%%%%%%%%%%%%

%%%%%%%%%%%%%%%%%%%%%%%%%%%%%%%%%%%%%%%%%%%%%%%%%%%%%%%%%%%%%%%%%%%%%%%%%%%%%%%%%%%%%%%%%%%%%%%%%%%%
%%%%%%%%%%%%%%%%%%%%%%%%%%%%%%%%%%%%%%%%%%%%%%%%%%%%%%%%%%%%%%%%%%%%%%%%%%%%%%%%%%%%%%%%%%%%%%%%%%%%
\begin{proof}[Proof of Proposition~\ref{2}]
(1)	For any generator $S_\alpha$ of $\pi$, the complement of $\alpha$ is homeomorphic to either
	\begin{enumerate}
	\item	$N_{g-1,n+1}$,
	\item	$N_{g-2,n+2}$,
	\item	$\Sigma_{h,n+r}$ if $g=2h+r$ for $r=1$, $2$,
	\item	$N_{h,m+1}\sqcup{N_{g-h,n-m+1}}$ for $1\leq{h}\leq{g-1}$ and $0\leq{m}\leq{n}$ or
	\item	$\Sigma_{h,m+1}\sqcup{N_{g-2h,n-m+1}}$ for $\displaystyle0\leq{h}\leq\frac{g-1}{2}$ and $0\leq{m}\leq{n}$
	\end{enumerate}
	(see \cite{S1}).
	Therefore, there is $f\in\PM(N_{g,n+1})$ such that $\alpha=f_\sharp(\beta)$, where $\beta$ is either one of the simple loops as in Figure~\ref{normal-position-loop-non-ori-surf}.
	For the case (a), we have $S_\alpha=S_{f_\sharp(x_1)}$.
	For the case (b), by the relation (2) of $\pi$, we have $S_\alpha=S_{f_\sharp(x_1)}S_{f_\sharp(x_2)}$.
	For the cases (c), by the relation (2) of $\pi$, we have $S_\alpha=S_{f_\sharp(x_1)}\cdots{}S_{f_\sharp(x_g)}$.
	For the case (d), by the relation (2) of $\pi$, we have $S_\alpha=S_{f_\sharp(x_1)}^2\cdots{}S_{f_\sharp(x_h)}^2S_{f_\sharp(y_{k_1})}\cdots{}S_{f_\sharp(y_{k_m})}$ for some $1\leq{k_1<\cdots<k_m}\leq{n}$.
	For the case (e), by the relations (1) and (2) of $\pi$, we have
	\begin{eqnarray*}
	S_\alpha
	&=&
	S_{f_\sharp(x_1)}\cdots{}S_{f_\sharp(x_{2h})}S_{f_\sharp(x_{2h+1})}^{-1}S_{f_\sharp(x_{2h})}^{-2}\cdots{}S_{f_\sharp(x_{2})}^{-2}S_{f_\sharp(x_{1})}^{-1}\\
	&&
	~\cdot~S_{f_\sharp(x_{2})}\cdots{}S_{f_\sharp(x_{2h+1})}S_{f_\sharp(y_{k_1})}\cdots{}S_{f_\sharp(y_{k_m})}
	\end{eqnarray*}
	for some $1\leq{k_1<\cdots<k_m}\leq{n}$ if $h\neq0$.
	If $h=0$, by the relation (2) of $\pi$, we have $S_\alpha=S_{f_\sharp(y_{k_1})}\cdots{}S_{f_\sharp(y_{k_m})}$ for some $1\leq{k_1<\cdots<k_m}\leq{n}$.
	Since each symbol of the right hand sides is in $X$, we conclude that $X$ generates $\pi$.
	%%%%%%%%%%%%%%%%%%%%%%%%%%%%%%%%%%%%%%%%%%%%%%%%%%%%%%%%%%%%%%%%%%%%%%%%%%%%%%%%%%%%%%%%%%%%%%%%%%%%
	\begin{figure}[htbp]
	\subfigure[A standard position oriented simple loop whose complement is homeomorphic to $N_{g-1,n+1}$.]{\includegraphics[scale=0.5]{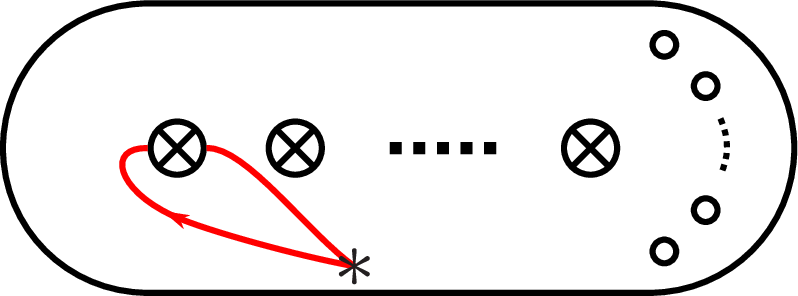}}
	\subfigure[A standard position oriented simple loop whose complement is homeomorphic to $N_{g-2,n+2}$.]{\includegraphics[scale=0.5]{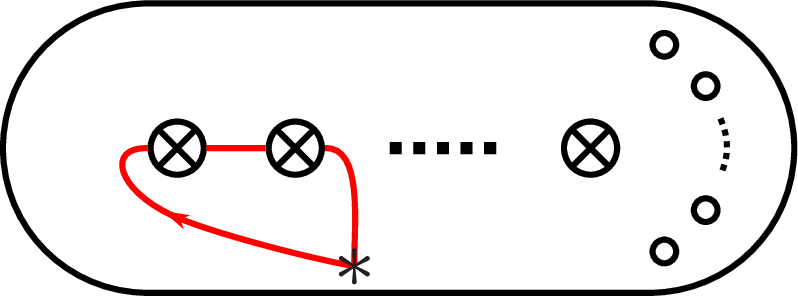}}
	\subfigure[A standard position oriented simple loop whose complement is homeomorphic to $\Sigma_{h,n+r}$ if $g=2h+r$ for $r=1$, $2$.]{\includegraphics[scale=0.5]{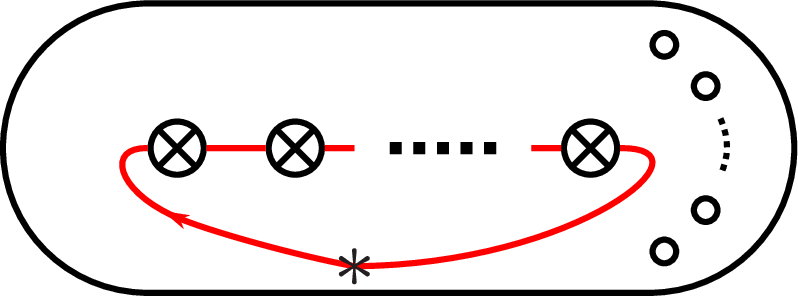}}
	\subfigure[A standard position oriented simple loop whose complement is homeomorphic to $N_{h,m+1}\sqcup{N_{g-h,n-m+1}}$ for $1\leq{h}\leq{g-1}$ and $0\leq{m}\leq{n}$.]{\includegraphics[scale=0.5]{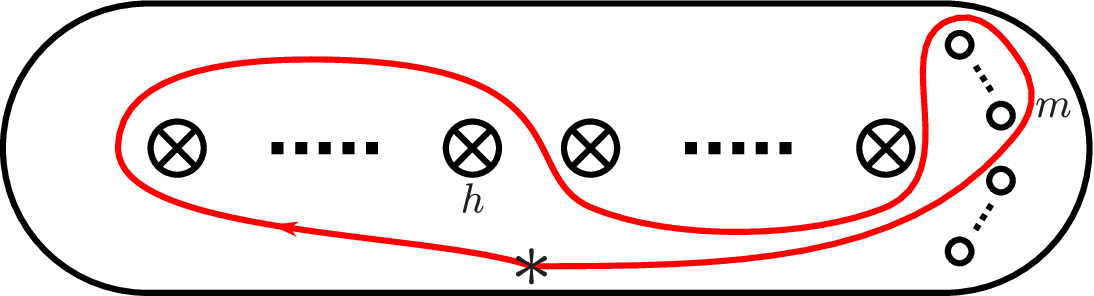}}
	\subfigure[A standard position oriented simple loop whose complement is homeomorphic to $\Sigma_{h,m+1}\sqcup{N_{g-2h,n-m+1}}$ for $\displaystyle0\leq{h}\leq\frac{g-1}{2}$ and $0\leq{m}\leq{n}$.]{\includegraphics[scale=0.5]{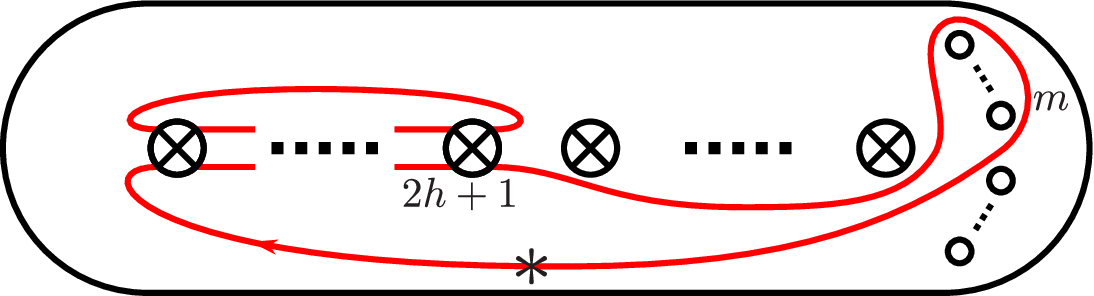}}
	\caption{}\label{normal-position-loop-non-ori-surf}
	\end{figure}
	%%%%%%%%%%%%%%%%%%%%%%%%%%%%%%%%%%%%%%%%%%%%%%%%%%%%%%%%%%%%%%%%%%%%%%%%%%%%%%%%%%%%%%%%%%%%%%%%%%%%
	
(2)	For any $S_\alpha\in{X}$, if $\alpha$ is a one-sided simple loop whose complement is non-orientable, there is $f\in\PM(N_{g,n+1})$ such that $f_\sharp(x_1)=\alpha$, and hence $f(S_{x_1})=S_\alpha$.
	If $\alpha$ is a separating simple loop which bounds the $m$-th boundary component for $1\leq{m}\leq{n-1}$, there is $f\in\PM(N_{g,n+1})$ such that $f_\sharp(y_m)=\alpha$, and hence $f(S_{y_m})=S_\alpha$.
	Therefore we obtain the claim.
	
(3)	In this proof, we omit details of calculations.
	
	Let $y=t_{a_i}$ for $1\leq{i}\leq{g}$.
	We calculate
	\begin{eqnarray*}
	y(S_{x_i})&=&S_{x_ix_{i+1}^{-1}x_i^{-1}}\overset{(1),(2)}{=}S_{x_i}S_{x_{i+1}}^{-1}S_{x_i}^{-1},\\
	y^{-1}(S_{x_i})&=&S_{x_i^2x_{i+1}}\overset{(2)}{=}S_{x_i}^2S_{x_{i+1}},\\
	y(S_{x_{i+1}})&=&S_{x_ix_{i+1}^2}\overset{(2)}{=}S_{x_i}S_{x_{i+1}}^2,\\
	y^{-1}(S_{x_{i+1}})&=&S_{x_{i+1}^{-1}x_i^{-1}x_{i+1}}\overset{(1),(2)}{=}S_{x_{i+1}}^{-1}S_{x_i}^{-1}S_{x_{i+1}}
	\end{eqnarray*}
	and $y^{\pm1}(x)=x$ for any other $x\in{X^\prime}$.
	
	Let $y=t_{b}$.
	We calculate
	\begin{eqnarray*}
	y(S_{x_1})&=&S_{x_1x_2x_3x_4^{-1}x_3^{-2}x_2^{-2}x_1^{-1}}\overset{(1),(2)}{=}S_{x_1}S_{x_2}S_{x_3}S_{x_4}^{-1}S_{x_3}^{-2}S_{x_2}^{-2}S_{x_1}^{-1},\\
	y^{-1}(S_{x_1})&=&S_{x_1^2x_2^2x_3^2x_4x_3^{-1}x_2^{-1}}\overset{(1),(2)}{=}S_{x_1}^2S_{x_2}^2S_{x_3}^2S_{x_4}S_{x_3}^{-1}S_{x_2}^{-1},\\
	y(S_{x_2})&=&S_{x_1x_2^2x_3^2x_4x_3^{-1}}\overset{(1),(2)}{=}S_{x_1}S_{x_2}^2S_{x_3}^2S_{x_4}S_{x_3}^{-1},\\
	y^{-1}(S_{x_2})&=&S_{x_2x_3x_4^{-1}x_3^{-2}x_2^{-2}x_1^{-1}x_2}\overset{(1),(2)}{=}S_{x_2}S_{x_3}S_{x_4}^{-1}S_{x_3}^{-2}S_{x_2}^{-2}S_{x_1}^{-1}S_{x_2},\\
	y(S_{x_3})&=&S_{x_3x_4^{-1}x_3^{-2}x_2^{-2}x_1^{-1}x_2x_3}\overset{(1),(2)}{=}S_{x_3}S_{x_4}^{-1}S_{x_3}^{-2}S_{x_2}^{-2}S_{x_1}^{-1}S_{x_2}S_{x_3},\\
	y^{-1}(S_{x_3})&=&S_{x_2^{-1}x_1x_2^2x_3^2x_4}\overset{(1),(2)}{=}S_{x_2}^{-1}S_{x_1}S_{x_2}^2S_{x_3}^2S_{x_4},\\
	y(S_{x_4})&=&S_{x_3^{-1}x_2^{-1}x_1x_2^2x_3^2x_4^2}\overset{(1),(2)}{=}S_{x_3}^{-1}S_{x_2}^{-1}S_{x_1}S_{x_2}^2S_{x_3}^2S_{x_4}^2,\\
	y^{-1}(S_{x_4})&=&S_{x_4^{-1}x_3^{-2}x_2^{-2}x_1^{-1}x_2x_3x_4}\overset{(1),(2)}{=}S_{x_4}^{-1}S_{x_3}^{-2}S_{x_2}^{-2}S_{x_1}^{-1}S_{x_2}S_{x_3}S_{x_4}
	\end{eqnarray*}
	and $y^{\pm1}(x)=x$ for any other $x\in{X^\prime}$.
	
	Let $y=Y_{\mu,a_1}$.
	We calculate
	\begin{eqnarray*}
	y(S_{x_1})&=&S_{x_1^2x_2x_1^{-1}x_2^{-1}x_1^{-2}}\overset{(1),(2)}{=}S_{x_1}^2S_{x_2}S_{x_1}^{-1}S_{x_2}^{-1}S_{x_1}^{-2},\\
	y^{-1}(S_{x_1})&=&S_{x_2^{-1}x_1^{-1}x_2}\overset{(1),(2)}{=}S_{x_2}^{-1}S_{x_1}^{-1}S_{x_2},\\
	y(S_{x_2})&=&S_{x_1^2x_2}\overset{(2)}{=}S_{x_1}^2S_{x_2},~
	y^{-1}(S_{x_2})=S_{x_2^{-1}x_1^2x_2^2}\overset{(1), (2)}{=}S_{x_2}^{-1}S_{x_1}^2S_{x_2}^2
	\end{eqnarray*}
	and $y^{\pm1}(x)=x$ for any other $x\in{X^\prime}$.
	
	Let $y=B_{r_k}$ for $1\leq{k}\leq{n}$.
	We calculate
	\begin{eqnarray*}
	y(S_{x_g})&=&S_{x_g^2y_kx_g^{-1}}\overset{(1),(2)}{=}S_{x_g}^2S_{y_k}S_{x_g}^{-1},~
	y^{-1}(S_{x_g})=S_{x_gy_k}\overset{(2)}{=}S_{x_g}S_{y_k},\\
	y(S_{y_k})&=&S_{x_gy_kx_g^{-1}}\overset{(1),(2)}{=}S_{x_g}S_{y_k}S_{x_g}^{-1},\\
	y^{-1}(S_{y_k})&=&S_{y_k^{-1}x_g^{-1}y_kx_gy_k}\overset{(1),(2)}{=}S_{y_k}^{-1}S_{x_g}^{-1}S_{y_k}S_{x_g}S_{y_k},\\
	y(S_{y_l})&=&S_{x_gy_k^{-1}x_g^{-1}y_k^{-1}y_ly_kx_gy_kx_g^{-1}}\overset{(1),(2)}{=}S_{x_g}S_{y_k}^{-1}S_{x_g}^{-1}S_{y_k}^{-1}S_{y_l}S_{y_k}S_{x_g}S_{y_k}S_{x_g}^{-1},\\
	y^{-1}(S_{y_l})&=&S_{y_k^{-1}x_g^{-1}y_k^{-1}x_gy_lx_g^{-1}y_kx_gy_k}\overset{(1),(2)}{=}S_{y_k}^{-1}S_{x_g}^{-1}S_{y_k}^{-1}S_{x_g}S_{y_l}S_{x_g}^{-1}S_{y_k}S_{x_g}S_{y_k}
	\end{eqnarray*}
	for $l<k$, and $y^{\pm1}(x)=x$ for any other $x\in{X^\prime}$.
	
	Let $y=B_{r_0}$.
	We calculate
	\begin{eqnarray*}
	y^{\pm1}(S_{x_j})&=&S_{x_g^{\mp1}x_jx_g^{\pm1}}\overset{(1), (2)}{=}S_{x_g}^{\mp1}S_{x_j}S_{x_g}^{\pm1},\\
	y^{\pm1}(S_{y_l})&=&S_{x_g^{\mp1}y_lx_g^{\pm1}}\overset{(1), (2)}{=}S_{x_g}^{\mp1}S_{y_l}S_{x_g}^{\pm1}
	\end{eqnarray*}
	for $1\leq{j}\leq{g}$ and $1\leq{l}\leq{n-1}$.
	
	Let $y=t_{s_{kl}}$ for $1\leq{k<l}\leq{n}$.
	We calculate
	\begin{eqnarray*}
	y(S_{y_{k}})&=&S_{y_ky_ly_ky_l^{-1}y_k^{-1}}\overset{(1),(2)}{=}S_{y_k}S_{y_l}S_{y_k}S_{y_l}^{-1}S_{y_k}^{-1},\\
	y^{-1}(S_{y_{k}})&=&S_{y_l^{-1}y_ky_l}\overset{(1),(2)}{=}S_{y_l}^{-1}S_{y_k}S_{y_l},\\
	y(S_{y_{l}})&=&S_{y_ky_ly_k^{-1}}\overset{(1),(2)}{=}S_{y_k}S_{y_l}S_{y_k}^{-1},\\
	y^{-1}(S_{y_{l}})&=&S_{y_l^{-1}y_k^{-1}y_ly_ky_l}\overset{(1),(2)}{=}S_{y_l}^{-1}S_{y_k}^{-1}S_{y_l}S_{y_k}S_{y_l},\\
	y(S_{y_{m}})&=&S_{[y_k,y_l]y_m[y_k,y_l]^{-1}}\overset{(1),(2)}{=}[S_{y_k},S_{y_l}]S_{y_m}[S_{y_k},S{y_l}]^{-1},\\
	y^{-1}(S_{y_{m}})&=&S_{[y_l^{-1},y_k^{-1}]y_m[y_l^{-1},y_k^{-1}]^{-1}}\overset{(1),(2)}{=}[S_{y_l}^{-1},S_{y_k}^{-1}]S_{y_m}[S_{y_l}^{-1},S_{y_k}^{-1}]^{-1}
	\end{eqnarray*}
	for $k<m<l$, and $y^{\pm1}(x)=x$ for any other $x\in{X^\prime}$.
		
	Hence we have that for any $x\in{X^\prime}$ and $y\in{Y}$, $y^{\pm1}(x)$ is in the subgroup of $\pi$ generated by $X^\prime$, by Lemma~\ref{y_n}.
\end{proof}
%%%%%%%%%%%%%%%%%%%%%%%%%%%%%%%%%%%%%%%%%%%%%%%%%%%%%%%%%%%%%%%%%%%%%%%%%%%%%%%%%%%%%%%%%%%%%%%%%%%%
%%%%%%%%%%%%%%%%%%%%%%%%%%%%%%%%%%%%%%%%%%%%%%%%%%%%%%%%%%%%%%%%%%%%%%%%%%%%%%%%%%%%%%%%%%%%%%%%%%%%

%%%%%%%%%%%%%%%%%%%%%%%%%%%%%%%%%%%%%%%%%%%%%%%%%%%%%%%%%%%%%%%%%%%%%%%%%%%%%%%%%%%%%%%%%%%%%%%%%%%%
%%%%%%%%%%%%%%%%%%%%%%%%%%%%%%%%%%%%%%%%%%%%%%%%%%%%%%%%%%%%%%%%%%%%%%%%%%%%%%%%%%%%%%%%%%%%%%%%%%%%
\begin{proof}[Proof of Theorem~\ref{main-1} of the case where $S$ is non-orientable]
By Lemma~\ref{main-lem} and Proposition~\ref{2}, it follows that $\pi$ is generated by $X^\prime$.
There is a natural map $\pi\to\pi_1(N_{g,n},\ast)$.
The relations~(1) and (2) of $\pi$ are satisfied in $\pi_1(N_{g,n},\ast)$ clearly.
Hence the map is a homomorphism.
In addition, the relation $S_{x_1}^2\cdots{}S_{x_g}^2=1$ is obtained from the relation~(2) of $\pi$ for $n=0$.
Therefore the map is an isomorphism for any $n\geq0$.
Thus we complete the proof.
\end{proof}
%%%%%%%%%%%%%%%%%%%%%%%%%%%%%%%%%%%%%%%%%%%%%%%%%%%%%%%%%%%%%%%%%%%%%%%%%%%%%%%%%%%%%%%%%%%%%%%%%%%%
%%%%%%%%%%%%%%%%%%%%%%%%%%%%%%%%%%%%%%%%%%%%%%%%%%%%%%%%%%%%%%%%%%%%%%%%%%%%%%%%%%%%%%%%%%%%%%%%%%%%

%%%%%%%%%%%%%%%%%%%%%%%%%%%%%%%%%%%%%%%%%%%%%%%%%%%%%%%%%%%%%%%%%%%%%%%%%%%%%%%%%%%%%%%%%%%%%%%%%%%%
%%%%%%%%%%%%%%%%%%%%%%%%%%%%%%%%%%%%%%%%%%%%%%%%%%%%%%%%%%%%%%%%%%%%%%%%%%%%%%%%%%%%%%%%%%%%%%%%%%%%
%%%%%%%%%%%%%%%%%%%%%%%%%%%%%%%%%%%%%%%%%%%%%%%%%%%%%%%%%%%%%%%%%%%%%%%%%%%%%%%%%%%%%%%%%%%%%%%%%%%%
%%%%%%%%%%%%%%%%%%%%%%%%%%%%%%%%%%%%%%%%%%%%%%%%%%%%%%%%%%%%%%%%%%%%%%%%%%%%%%%%%%%%%%%%%%%%%%%%%%%%
%%%%%%%%%%%%%%%%%%%%%%%%%%%%%%%%%%%%%%%%%%%%%%%%%%%%%%%%%%%%%%%%%%%%%%%%%%%%%%%%%%%%%%%%%%%%%%%%%%%%
%%%%%%%%%%%%%%%%%%%%%%%%%%%%%%%%%%%%%%%%%%%%%%%%%%%%%%%%%%%%%%%%%%%%%%%%%%%%%%%%%%%%%%%%%%%%%%%%%%%%
%%%%%%%%%%%%%%%%%%%%%%%%%%%%%%%%%%%%%%%%%%%%%%%%%%%%%%%%%%%%%%%%%%%%%%%%%%%%%%%%%%%%%%%%%%%%%%%%%%%%
%%%%%%%%%%%%%%%%%%%%%%%%%%%%%%%%%%%%%%%%%%%%%%%%%%%%%%%%%%%%%%%%%%%%%%%%%%%%%%%%%%%%%%%%%%%%%%%%%%%%
%%%%%%%%%%%%%%%%%%%%%%%%%%%%%%%%%%%%%%%%%%%%%%%%%%%%%%%%%%%%%%%%%%%%%%%%%%%%%%%%%%%%%%%%%%%%%%%%%%%%
%%%%%%%%%%%%%%%%%%%%%%%%%%%%%%%%%%%%%%%%%%%%%%%%%%%%%%%%%%%%%%%%%%%%%%%%%%%%%%%%%%%%%%%%%%%%%%%%%%%%
\section{Proof of Theorem~\ref{main-2}}\label{pi^+}

Let $x_{ij}=x_ix_j$ and $z_k=x_gy_kx_g^{-1}$ for $1\leq{i,j}\leq{g}$ and $1\leq{k}\leq{n-1}$, where $x_1,\dots,x_g$ and $y_1,\dots,y_{n-1}$ are simple loops of $N_{g,n}$ as shown in Figure~\ref{gen-pi_1-non-ori-surf}.
We first consider a presentation for $\pi_1^+(N_{g,n},\ast)$ as follows.

%%%%%%%%%%%%%%%%%%%%%%%%%%%%%%%%%%%%%%%%%%%%%%%%%%%%%%%%%%%%%%%%%%%%%%%%%%%%%%%%%%%%%%%%%%%%%%%%%%%%
%%%%%%%%%%%%%%%%%%%%%%%%%%%%%%%%%%%%%%%%%%%%%%%%%%%%%%%%%%%%%%%%%%%%%%%%%%%%%%%%%%%%%%%%%%%%%%%%%%%%
\begin{lem}
$\pi_1^+(N_{g,n},\ast)$ is the free group freely generated by $x_{12},\dots,x_{g-1\:g}$, $x_{11},\dots,x_{gg}$, $y_1,\dots,y_{n-1}$ and $z_1,\dots,z_{n-1}$ for $n\geq1$, and the group generated by $x_{12},\dots,x_{g-1\:g}$ and $x_{11},\dots,x_{gg}$ which has two relations $x_{11}\cdots{}x_{gg}=1$ and $x_{gg}x_{g-1\:g}^{-1}x_{g-1\:g-1}x_{g-2\:g-1}^{-1}\cdots{}x_{22}x_{12}^{-1}x_{11}x_{12}\cdots{}x_{g-1\:g}=1$ for $n=0$.
\end{lem}
%%%%%%%%%%%%%%%%%%%%%%%%%%%%%%%%%%%%%%%%%%%%%%%%%%%%%%%%%%%%%%%%%%%%%%%%%%%%%%%%%%%%%%%%%%%%%%%%%%%%
%%%%%%%%%%%%%%%%%%%%%%%%%%%%%%%%%%%%%%%%%%%%%%%%%%%%%%%%%%%%%%%%%%%%%%%%%%%%%%%%%%%%%%%%%%%%%%%%%%%%

%%%%%%%%%%%%%%%%%%%%%%%%%%%%%%%%%%%%%%%%%%%%%%%%%%%%%%%%%%%%%%%%%%%%%%%%%%%%%%%%%%%%%%%%%%%%%%%%%%%%
%%%%%%%%%%%%%%%%%%%%%%%%%%%%%%%%%%%%%%%%%%%%%%%%%%%%%%%%%%%%%%%%%%%%%%%%%%%%%%%%%%%%%%%%%%%%%%%%%%%%
\begin{proof}
It is known that $\pi_1^+(N_{g,n},\ast)$ is an index two subgroup of $\pi_1(N_{g,n},\ast)$ (see \cite{KO1}).
Hence we can obtain a presentation of $\pi_1^+(N_{g,n},\ast)$ by the Reidemeister Schreier method (for details, for instance see \cite{J}).
Note that $\pi_1(N_{g,n},\ast)$ is generated by $x_1,\dots,x_g$ and $y_1,\dots,y_{n-1}$.
We chose $\{1,x_g\}$ as a Schreier transversal for $\pi_1^+(N_{g,n},\ast)$ in $\pi_1(N_{g,n},\ast)$.
Then it follows that $\pi_1^+(N_{g,n},\ast)$ is generated by $x_1x_g^{-1},\dots,x_{g-1}x_g^{-1}$, $x_gx_1,\dots,x_gx_g$, $y_1,\dots,y_{n-1}$ and $z_1,\dots,z_{n-1}$ (see \cite{Kob}).
In addition, we see
\begin{eqnarray*}
1(x_1^2\cdots{}x_g^2)1^{-1}
&=&
x_1x_g^{-1}\cdot{}x_gx_1\cdot{}x_2x_g^{-1}\cdot{}x_gx_2\cdots{}x_{g-1}x_g^{-1}\cdot{}x_gx_{g-1}\cdot{}x_gx_g,\\
x_g(x_1^2\cdots{}x_g^2)x_g^{-1}
&=&
x_gx_1\cdot{}x_1x_g^{-1}\cdot{}x_gx_2\cdot{}x_2x_g^{-1}\cdots{}x_gx_{g-1}\cdot{}x_{g-1}x_g^{-1}\cdot{}x_gx_g.
\end{eqnarray*}
Hence when $n=0$, we have two relations
\begin{eqnarray*}
&&x_1x_g^{-1}\cdot{}x_gx_1\cdot{}x_2x_g^{-1}\cdot{}x_gx_2\cdots{}x_{g-1}x_g^{-1}\cdot{}x_gx_{g-1}\cdot{}x_gx_g=1,\\
&&x_gx_1\cdot{}x_1x_g^{-1}\cdot{}x_gx_2\cdot{}x_2x_g^{-1}\cdots{}x_gx_{g-1}\cdot{}x_{g-1}x_g^{-1}\cdot{}x_gx_g=1.
\end{eqnarray*}

Let $G$ be the group which has the presentation of the lemma.
We next show that $G$ is isomorphic to $\pi_1^+(N_{g,n},\ast)$.
Let $\varphi:G\to\pi_1^+(N_{g,n},\ast)$ and $\psi:\pi_1^+(N_{g,n},\ast)\to{}G$ be homomorphisms defined as
\begin{eqnarray*}
&&
\varphi(x_{i\:i+1})=x_ix_g^{-1}\cdot{}x_gx_{i+1},~
\varphi(x_{jj})=x_jx_g^{-1}\cdot{}x_gx_j,~
\varphi(y_k)=y_k,~
\varphi(z_k)=z_k,\\
&&
\psi(x_ix_g^{-1})=x_{i\:i+1}x_{i+1\:i+1}^{-1}x_{i+1\:i+2}x_{i+2\:i+2}^{-1}\cdots{}x_{g-1\:g}x_{gg}^{-1},\\
&&
\psi(x_gx_j)=x_{gg}x_{g-1\:g}^{-1}x_{g-1\:g-1}x_{g-2\:g-1}^{-1}\cdots{}x_{j+1\:j+1}x_{j\:j+1}^{-1}x_{jj},\\
&&
\psi(y_k)=y_k,~
\psi(z_k)=z_k
\end{eqnarray*}
for $1\leq{i}\leq{g-1}$, $1\leq{j}\leq{g}$ and $1\leq{k}\leq{n-1}$.
We calculate
\begin{eqnarray*}
&&\varphi(x_{11}\cdots{}x_{g-1\:g-1}x_{gg})=x_1x_g^{-1}\cdot{}x_gx_1\cdots{}x_{g-1}x_g^{-1}\cdot{}x_gx_{g-1}\cdot{}x_gx_g,\\
&&
\varphi(x_{gg}x_{g-1\:g}^{-1}x_{g-1\:g-1}x_{g-2\:g-1}^{-1}\cdots{}x_{22}x_{12}^{-1}x_{11}x_{12}\cdots{}x_{g-1\:g})\\
&=&
x_gx_g(x_{g-1}x_g^{-1}\cdot{}x_gx_g)^{-1}(x_{g-1}x_g^{-1}\cdot{}x_gx_{g-1})(x_{g-2}x_g^{-1}\cdot{}x_gx_{g-1})^{-1}\\
&&
\cdots(x_2x_g^{-1}\cdot{}x_gx_2)(x_1x_g^{-1}\cdot{}x_gx_2)^{-1}(x_1x_g^{-1}\cdot{}x_gx_1)(x_1x_g^{-1}\cdot{}x_gx_2)\\
&&
\cdots(x_{g-1}x_g^{-1}\cdot{}x_gx_g)\\
&=&
x_gx_1\cdot{}x_1x_g^{-1}\cdot{}x_gx_2\cdot{}x_2x_g^{-1}\cdots{}x_gx_{g-1}\cdot{}x_{g-1}x_g^{-1}\cdot{}x_gx_g,\\
&&
\psi(x_1x_g^{-1}\cdot{}x_gx_1\cdot{}x_2x_g^{-1}\cdot{}x_gx_2\cdots{}x_{g-1}x_g^{-1}\cdot{}x_gx_{g-1}\cdot{}x_gx_g)\\
&=&
x_{11}x_{22}\cdots{}x_{g-1\:g-1}x_{gg},\\
&&
\psi(x_gx_1\cdot{}x_1x_g^{-1}\cdot{}x_gx_2\cdot{}x_2x_g^{-1}\cdots{}x_gx_{g-1}\cdot{}x_{g-1}x_g^{-1}\cdot{}x_gx_g)\\
&=&
x_{gg}x_{g-1\:g}^{-1}x_{g-1\:g-1}x_{g-2\:g-1}^{-1}\cdots{}x_{22}x_{12}^{-1}x_{11}x_{12}\cdots{}x_{g-1\:g}.
\end{eqnarray*}
Hence $\varphi$ and $\psi$ are well defined even if $n=0$.
In addition, we have
\begin{eqnarray*}
\psi\varphi(x_{i\:i+1})
&=&
\psi(x_ix_g^{-1}\cdot{}x_gx_{i+1})\\
&=&
x_{i\:i+1}x_{i+1\:i+1}^{-1}\cdots{}x_{g-1\:g}x_{gg}^{-1}\\
&&\cdot{}x_{gg}x_{g-1\:g}^{-1}\cdots{}x_{i+2\:i+2}x_{i+1\:i+2}^{-1}x_{i+1\:i+1}
=x_{i\:i+1},\\
\psi\varphi(x_{jj})
&=&
\psi(x_jx_g^{-1}\cdot{}x_gx_{j})\\
&=&
x_{j\:j+1}x_{j+1\:j+1}^{-1}\cdots{}x_{g-1\:g}x_{gg}^{-1}\cdot{}x_{gg}x_{g-1\:g}^{-1}\cdots{}x_{j+1\:j+1}x_{j\:j+1}^{-1}x_{jj}\\
&=&
x_{jj},\\
\psi\varphi(y_k)
&=&
\psi(y_k)=y_k,~
\psi\varphi(z_k)=\psi(z_k)=z_k,\\
\varphi\psi(x_ix_g^{-1})
&=&
\varphi(x_{i\:i+1}x_{i+1\:i+1}^{-1}\cdots{}x_{g-1\:g}x_{gg}^{-1})\\
&=&
(x_ix_g^{-1}\cdot{}x_gx_{i+1})(x_{i+1}x_g^{-1}\cdot{}x_gx_{i+1})^{-1}\\
&&\cdots(x_{g-1}x_g^{-1}\cdot{}x_gx_g)(x_gx_g)^{-1}
=x_ix_g^{-1},\\
\varphi\psi(x_gx_j)
&=&
\varphi(x_{gg}x_{g-1\:g}^{-1}\cdots{}x_{j+1\:j+1}x_{j\:j+1}^{-1}x_{jj})\\
&=&
(x_gx_g)(x_{g-1}x_g^{-1}\cdot{}x_gx_g)^{-1}\cdots{}(x_{j+1}x_g^{-1}\cdot{}x_gx_{j+1})\\
&&\cdot(x_jx_g^{-1}\cdot{}x_gx_{j+1})^{-1}(x_jx_g^{-1}\cdot{}x_gx_j)
=x_gx_j,\\
\varphi\psi(y_k)
&=&
\varphi(y_k)=y_k,~
\varphi\psi(z_k)=\varphi(z_k)=z_k.
\end{eqnarray*}
Therefore $\varphi$ and $\psi$ are the isomorphisms.
Thus we finish the proof.
\end{proof}
%%%%%%%%%%%%%%%%%%%%%%%%%%%%%%%%%%%%%%%%%%%%%%%%%%%%%%%%%%%%%%%%%%%%%%%%%%%%%%%%%%%%%%%%%%%%%%%%%%%%
%%%%%%%%%%%%%%%%%%%%%%%%%%%%%%%%%%%%%%%%%%%%%%%%%%%%%%%%%%%%%%%%%%%%%%%%%%%%%%%%%%%%%%%%%%%%%%%%%%%%

Let $X$ be a set consisting of $S_\alpha$, where $\alpha$ is a non-separating two-sided simple loop whose complement is non-orientable, or a separating simple loop which bounds the $m$-th boundary component for $1\leq{m}\leq{n-1}$ or one crosscap whose complement is non-orientable, and let $X^\prime$ be the following subset of $X$:
$$X^\prime=\{S_{x_{12}},\dots,S_{x_{g-1\:g}},S_{x_{11}},\dots,S_{x_{gg}},S_{y_1},\dots,S_{y_{n-1}},S_{z_1},\dots,S_{z_{n-1}}\}.$$
Let $Y$ be the generating set for $\PM(N_{g,n+1})$ given in Theorem~\ref{gen-PMF}.
In the actions on $\pi_1^+(N_{g,n},\ast)$ and $\pi^+$ by $\PM(N_{g,n+1})$, we regard the $(n+1)$-st boundary component of $N_{g,n+1}$ as $\ast$.
Recall that the action $f(S_\alpha)$ of $f\in\PM(N_{g,n+1})$ on $S_\alpha\in\pi^+$ was defined in Subsection~\ref{non-ori}.
We prove the following proposition.

%%%%%%%%%%%%%%%%%%%%%%%%%%%%%%%%%%%%%%%%%%%%%%%%%%%%%%%%%%%%%%%%%%%%%%%%%%%%%%%%%%%%%%%%%%%%%%%%%%%%
%%%%%%%%%%%%%%%%%%%%%%%%%%%%%%%%%%%%%%%%%%%%%%%%%%%%%%%%%%%%%%%%%%%%%%%%%%%%%%%%%%%%%%%%%%%%%%%%%%%%
\begin{prop}\label{3}
\begin{enumerate}
\item	$X$ generates $\pi^+$.
\item	$\PM(N_{g,n+1})(X^\prime)=X$.
\item	For any $x\in{X^\prime}$ and $y\in{Y}$, $y^{\pm1}(x)$ is in the subgroup of $\pi^+$ generated by $X^\prime$.
\end{enumerate}
\end{prop}
%%%%%%%%%%%%%%%%%%%%%%%%%%%%%%%%%%%%%%%%%%%%%%%%%%%%%%%%%%%%%%%%%%%%%%%%%%%%%%%%%%%%%%%%%%%%%%%%%%%%
%%%%%%%%%%%%%%%%%%%%%%%%%%%%%%%%%%%%%%%%%%%%%%%%%%%%%%%%%%%%%%%%%%%%%%%%%%%%%%%%%%%%%%%%%%%%%%%%%%%%

In order to prove the proposition, we show the following lemma.

%%%%%%%%%%%%%%%%%%%%%%%%%%%%%%%%%%%%%%%%%%%%%%%%%%%%%%%%%%%%%%%%%%%%%%%%%%%%%%%%%%%%%%%%%%%%%%%%%%%%
%%%%%%%%%%%%%%%%%%%%%%%%%%%%%%%%%%%%%%%%%%%%%%%%%%%%%%%%%%%%%%%%%%%%%%%%%%%%%%%%%%%%%%%%%%%%%%%%%%%%
\begin{lem}\label{ij}
$S_{x_{ij}}$, $S_{x_{ji}}$, $S_{y_n}$ and $S_{z_n}$ are in the subgroup of $\pi^+$ generated by $X^\prime$ for $1\leq{i<j}\leq{g}$, where $y_n$ is a simple loop of $N_{g,n}$ as shown in Figure~\ref{gen-pi_1-non-ori-surf} and $z_n=x_gy_nx_g^{-1}$.
\end{lem}
%%%%%%%%%%%%%%%%%%%%%%%%%%%%%%%%%%%%%%%%%%%%%%%%%%%%%%%%%%%%%%%%%%%%%%%%%%%%%%%%%%%%%%%%%%%%%%%%%%%%
%%%%%%%%%%%%%%%%%%%%%%%%%%%%%%%%%%%%%%%%%%%%%%%%%%%%%%%%%%%%%%%%%%%%%%%%%%%%%%%%%%%%%%%%%%%%%%%%%%%%

%%%%%%%%%%%%%%%%%%%%%%%%%%%%%%%%%%%%%%%%%%%%%%%%%%%%%%%%%%%%%%%%%%%%%%%%%%%%%%%%%%%%%%%%%%%%%%%%%%%%
%%%%%%%%%%%%%%%%%%%%%%%%%%%%%%%%%%%%%%%%%%%%%%%%%%%%%%%%%%%%%%%%%%%%%%%%%%%%%%%%%%%%%%%%%%%%%%%%%%%%
\begin{proof}
For $1\leq{i<j}\leq{g}$, if $j-i=1$, then $x_{ij}$ is in $X^\prime$ clearly.
If $j-i\geq2$, we calculate
\begin{eqnarray*}
S_{x_{ij}}&=&S_{x_{i\:j-1}x_{j-1\:j-1}^{-1}x_{j-1\:j}}
\overset{(2)}{=}S_{x_{i\:j-1}x_{j-1\:j}}S_{x_{j-1\:j}^{-1}x_{j-1\:j-1}^{-1}x_{j-1\:j}}\\
&\overset{(1),(2),(3)}{=}&S_{x_{i\:j-1}}S_{x_{j-1\:j}}S_{x_{j-1\:j}}^{-1}S_{x_{j-1\:j-1}}^{-1}S_{x_{j-1\:j}}
=S_{x_{i\:j-1}}S_{x_{j-1\:j-1}}^{-1}S_{x_{j-1\:j}}.
\end{eqnarray*}
By induction on $j-i$, it follows that $S_{x_{ij}}$ is in the subgroup of $\pi^+$ generated by $X^\prime$.
In addition, we calculate
\begin{eqnarray*}
S_{x_{ji}}&=&S_{x_{jj}x_{ij}^{-1}x_{ii}}
\overset{(2)}{=}S_{x_{ij}^{-1}}S_{x_{ij}x_{jj}x_{ij}^{-1}x_{ii}}
\overset{(1),(2)}{=}S_{x_{ij}}^{-1}S_{x_{ij}x_{jj}x_{ij}^{-1}}S_{x_{ii}}\\
&\overset{(3)}{=}&S_{x_{ij}}^{-1}S_{x_{ij}}S_{x_{jj}}S_{x_{ij}}^{-1}S_{x_{ii}}
=S_{x_{jj}}S_{x_{ij}}^{-1}S_{x_{ii}}.
\end{eqnarray*}
Hence $S_{x_{ji}}$ is also in the subgroup of $\pi^+$ generated by $X^\prime$.
Moreover, by the relations (1) and (2) of $\pi^+$, we calculate
\begin{eqnarray*}
S_{y_n}&=&S_{(x_{11}\cdots{}x_{gg}y_1\cdots{}y_{n-1})^{-1}}
=(S_{x_{11}}\cdots{}S_{x_{gg}}S_{y_1}\cdots{}S_{y_{n-1}})^{-1},\\
S_{z_n}&=&S_{(x_{g1}x_{12}x_{23}\cdots{}x_{g-1\:g}z_1\cdots{}z_{n-1})^{-1}}\\
&=&(S_{x_{g1}}S_{x_{12}}S_{x_{23}}\cdots{}S_{x_{g-1\:g}}S_{z_1}\cdots{}S_{z_{n-1}})^{-1}.
\end{eqnarray*}
Therefore $S_{y_n}$ and $S_{z_n}$ are also in the subgroup of $\pi^+$ generated by $X^\prime$.

Thus we get the claim.
\end{proof}
%%%%%%%%%%%%%%%%%%%%%%%%%%%%%%%%%%%%%%%%%%%%%%%%%%%%%%%%%%%%%%%%%%%%%%%%%%%%%%%%%%%%%%%%%%%%%%%%%%%%
%%%%%%%%%%%%%%%%%%%%%%%%%%%%%%%%%%%%%%%%%%%%%%%%%%%%%%%%%%%%%%%%%%%%%%%%%%%%%%%%%%%%%%%%%%%%%%%%%%%%

%%%%%%%%%%%%%%%%%%%%%%%%%%%%%%%%%%%%%%%%%%%%%%%%%%%%%%%%%%%%%%%%%%%%%%%%%%%%%%%%%%%%%%%%%%%%%%%%%%%%
%%%%%%%%%%%%%%%%%%%%%%%%%%%%%%%%%%%%%%%%%%%%%%%%%%%%%%%%%%%%%%%%%%%%%%%%%%%%%%%%%%%%%%%%%%%%%%%%%%%%
\begin{proof}[Proof of Proposition~\ref{3}]

(1)	For any generator $S_\alpha$ of $\pi^+$, the complement of $\alpha$ is homeomorphic to either
	\begin{enumerate}
	\item[(b)]	$N_{g-2,n+2}$,
	\item[(c)]	$\Sigma_{h,n+2}$ only if $g=2h+2$,
	\item[(d)]	$N_{h,m+1}\sqcup{N_{g-h,n-m+1}}$ for $1\leq{h}\leq{g-1}$ and $0\leq{m}\leq{n}$ or
	\item[(e)]	$\Sigma_{h,m+1}\sqcup{N_{g-2h,n-m+1}}$ for $\displaystyle0\leq{h}\leq\frac{g-1}{2}$ and $0\leq{m}\leq{n}$.
	\end{enumerate}
	(see \cite{S1}).
	Therefore, there is $f\in\PM(N_{g,n+1})$ such that $\alpha=f_\sharp(\beta)$, where $\beta$ is either one of the simple loops as in Figure~\ref{normal-position-loop-non-ori-surf}~(b), (c), (d) and (e).
	For the case (b), we have $S_\alpha=S_{f_\sharp(x_{12})}$.
	For the case (c), by the relation (2) of $\pi^+$, we have $S_\alpha=S_{f_\sharp(x_{12})}\cdots{}S_{f_\sharp(x_{g-1\:g})}$.
	For the case (d), by the relation (2) of $\pi^+$, we have $S_\alpha=S_{f_\sharp(x_{11})}\cdots{}S_{f_\sharp(x_{hh})}S_{f_\sharp(y_{k_1})}\cdots{}S_{f_\sharp(y_{k_m})}$ for some $1\leq{k_1<\cdots<k_m}\leq{n}$.
	For the case (e), by the relation (1) and (2) of $\pi^+$, we have
	\begin{eqnarray*}
	S_\alpha
	&=&
	S_{f_\sharp(x_{12})}S_{f_\sharp(x_{34})}\cdots{}S_{f_\sharp(x_{2h-1\:2h})}\cdot{}S_{f_\sharp(x_{2h\:2h+1})}^{-1}S_{f_\sharp(x_{2h-1\:2h})}^{-1}\cdots{}S_{f_\sharp(x_{12})}^{-1}\\
	&&
	~\cdot~S_{f_\sharp(x_{23})}S_{f_\sharp(x_{45})}\cdots{}S_{f_\sharp(x_{2h\:2h+1})}S_{f_\sharp(y_{k_1})}\cdots{}S_{f_\sharp(y_{k_m})}
	\end{eqnarray*}
	for some $1\leq{k_1<\cdots<k_m}\leq{n}$ if $h\neq0$.
	If $h=0$, by the relation (2) of $\pi^+$, we have $S_\alpha=S_{f_\sharp(y_{k_1})}\cdots{}S_{f_\sharp(y_{k_m})}$ for some $1\leq{k_1<\cdots<k_m}\leq{n}$.
	Since each symbol of the right hand sides is in $X$, we conclude that $X$ generates $\pi$.

(2)	For any $S_\alpha\in{X}$, if $\alpha$ is a non-separating two-sided simple loop whose complement is non-orientable, there is $f\in\PM(N_{g,n+1})$ such that $f_\sharp(x_{12})=\alpha$, and hence $f(S_{x_{12}})=S_\alpha$.
	If $\alpha$ is a separating simple loop which bounds the $m$-th boundary component for $1\leq{m}\leq{n-1}$, there is $f\in\PM(N_{g,n+1})$ such that $f_\sharp(y_m)=\alpha$, and hence $f(S_{y_m})=S_\alpha$.
	If $\alpha$ is a separating simple loop which bounds one crosscap whose complement is non-orientable, there is $f\in\PM(N_{g,n+1})$ such that $f_\sharp(x_{11})=\alpha$, and hence $f(S_{x_{11}})=S_\alpha$.
	Therefore we obtain the claim.

(3)	In this proof, we omit details of calculations.
	In calculations, we use the relation (3) as little as possible (see Remark~\ref{reduced-rel}).
	
	Let $y=t_{a_i}$ for $1\leq{i}\leq{g}$.
	We calculate
	\begin{eqnarray*}
	y(S_{x_{i-1\:i}})&=&S_{x_{i-1\:i}x_{i\:i+1}^{-1}}
	\overset{(1),(2)}{=}S_{x_{i-1\:i}}S_{x_{i\:i+1}}^{-1},
	\end{eqnarray*}
	\begin{eqnarray*}
	y^{-1}(S_{x_{i-1\:i}})&=&S_{x_{i-1\:i}x_{i\:i+1}}
	\overset{(2)}{=}S_{x_{i-1\:i}}S_{x_{i\:i+1}},
	\end{eqnarray*}
	\begin{eqnarray*}
	y(S_{x_{i+1\:i+2}})&=&S_{x_{i\:i+1}x_{i+1\:i+2}}
	\overset{(2)}{=}S_{x_{i\:i+1}}S_{x_{i+1\:i+2}},
	\end{eqnarray*}
	\begin{eqnarray*}
	y^{-1}(S_{x_{i+1\:i+2}})&=&S_{x_{i\:i+1}^{-1}x_{i+1\:i+2}}
	\overset{(1),(2)}{=}S_{x_{i\:i+1}}^{-1}S_{x_{i+1\:i+2}},
	\end{eqnarray*}
	\begin{eqnarray*}
	y(S_{x_{ii}})&=&S_{x_{i\:i+1}x_{i+1\:i+1}^{-1}x_{i\:i+1}^{-1}}
	\overset{(1),(3)}{=}S_{x_{i\:i+1}}S_{x_{i+1\:i+1}}^{-1}S_{x_{i\:i+1}}^{-1},
	\end{eqnarray*}
	\begin{eqnarray*}
	y^{-1}(S_{x_{ii}})&=&S_{x_{ii}x_{i+1\:i+1}x_{i\:i+1}^{-1}x_{ii}x_{i\:i+1}}
	\overset{(2)}{=}S_{x_{ii}x_{i+1\:i+1}}S_{x_{i\:i+1}^{-1}x_{ii}x_{i\:i+1}}\\
	&\overset{(1),(2),(3)}{=}&S_{x_{ii}}S_{x_{i+1\:i+1}}S_{x_{i\:i+1}}^{-1}S_{x_{ii}}S_{x_{i\:i+1}},
	\end{eqnarray*}
	\begin{eqnarray*}
	y(S_{x_{i+1\:i+1}})&=&S_{x_{i\:i+1}x_{i+1\:i+1}x_{i\:i+1}^{-1}x_{ii}x_{i+1\:i+1}}\\
	&\overset{(2)}{=}&S_{x_{i\:i+1}x_{i+1\:i+1}x_{i\:i+1}^{-1}}S_{x_{ii}x_{i+1\:i+1}}\\
	&\overset{(1),(2),(3)}{=}&S_{x_{i\:i+1}}S_{x_{i+1\:i+1}}S_{x_{i\:i+1}}^{-1}S_{x_{ii}}S_{x_{i+1\:i+1}},
	\end{eqnarray*}
	\begin{eqnarray*}
	y^{-1}(S_{x_{i+1\:i+1}})&=&S_{x_{i\:i+1}^{-1}x_{ii}^{-1}x_{i\:i+1}}
	\overset{(1),(3)}{=}S_{x_{i\:i+1}}^{-1}S_{x_{ii}}^{-1}S_{x_{i\:i+1}},
	\end{eqnarray*}
	\begin{eqnarray*}
	t_{a_{g-1}}^{\pm1}(S_{z_l})&=&S_{x_{g-1\:g}^{\pm1}z_lx_{g-1\:g}^{\mp1}}
	\overset{(1),(2)}{=}S_{x_{g-1\:g}}^{\pm1}S_{z_l}S_{x_{g-1\:g}}^{\mp1}
	\end{eqnarray*}
	for $1\leq{l}\leq{n-1}$, and $y^{\pm1}(x)=x$ for any other $x\in{X^\prime}$.

	Let $y=t_{b}$.
	We calculate
	\begin{eqnarray*}
	y(S_{x_{45}})&=&S_{x_{23}^{-1}x_{12}x_{23}x_{34}x_{45}}
	\overset{(1),(2)}{=}S_{x_{23}}^{-1}S_{x_{12}}S_{x_{23}}S_{x_{34}}S_{x_{45}},
	\end{eqnarray*}
	\begin{eqnarray*}
	y^{-1}(S_{x_{45}})&=&S_{x_{34}^{-1}x_{23}^{-1}x_{12}^{-1}x_{23}x_{45}}
	\overset{(1),(2)}{=}S_{x_{34}}^{-1}S_{x_{23}}^{-1}S_{x_{12}}^{-1}S_{x_{23}}S_{x_{45}},
	\end{eqnarray*}
	\begin{eqnarray*}
	y(S_{x_{11}})&=&S_{x_{12}x_{34}x_{44}^{-1}x_{33}^{-1}x_{12}^{-1}x_{13}x_{34}^{-1}x_{23}^{-1}x_{12}^{-1}}
	\overset{(2)}{=}S_{x_{12}x_{34}x_{44}^{-1}x_{33}^{-1}x_{12}^{-1}}S_{x_{13}x_{34}^{-1}x_{23}^{-1}x_{12}^{-1}}\\
	&\overset{(1),(2)}{=}&S_{x_{12}}S_{x_{34}}S_{x_{44}}^{-1}S_{x_{33}}^{-1}S_{x_{12}}^{-1}S_{x_{13}}S_{x_{34}}^{-1}S_{x_{23}}^{-1}S_{x_{12}}^{-1},
	\end{eqnarray*}
	\begin{eqnarray*}
	y^{-1}(S_{x_{11}})&=&S_{x_{11}x_{22}x_{33}x_{44}x_{34}^{-1}x_{12}^{-1}x_{11}x_{12}x_{23}x_{34}x_{23}^{-1}}\\
	&\overset{(2)}{=}&S_{x_{11}x_{22}x_{33}x_{44}}S_{x_{34}^{-1}x_{12}^{-1}x_{11}x_{12}x_{23}x_{34}x_{23}^{-1}}\\
	&\overset{(1),(2)}{=}&S_{x_{11}}S_{x_{22}}S_{x_{33}}S_{x_{44}}S_{x_{34}}^{-1}S_{x_{12}^{-1}x_{11}x_{12}x_{23}x_{34}x_{23}^{-1}}\\
	&\overset{(2)}{=}&S_{x_{11}}S_{x_{22}}S_{x_{33}}S_{x_{44}}S_{x_{34}}^{-1}S_{x_{12}^{-1}x_{11}x_{12}}S_{x_{23}x_{34}x_{23}^{-1}}\\
	&\overset{(1),(2),(3)}{=}&S_{x_{11}}S_{x_{22}}S_{x_{33}}S_{x_{44}}S_{x_{34}}^{-1}S_{x_{12}}^{-1}S_{x_{11}}S_{x_{12}}S_{x_{23}}S_{x_{34}}S_{x_{23}}^{-1},
	\end{eqnarray*}
	\begin{eqnarray*}
	y(S_{x_{22}})&=&S_{x_{12}x_{23}x_{34}x_{13}^{-1}x_{11}x_{22}x_{33}x_{44}x_{34}^{-1}}
	\overset{(2)}{=}S_{x_{12}x_{23}}S_{x_{34}x_{13}^{-1}x_{11}x_{22}x_{33}x_{44}x_{34}^{-1}}\\
	&\overset{(1),(2)}{=}&S_{x_{12}}S_{x_{23}}S_{x_{34}}S_{x_{13}}^{-1}S_{x_{11}x_{22}x_{33}x_{44}x_{34}^{-1}}\\
	&\overset{(2)}{=}&S_{x_{12}}S_{x_{23}}S_{x_{34}}S_{x_{13}}^{-1}S_{x_{11}}S_{x_{22}}S_{x_{33}x_{44}x_{34}^{-1}}\\
	&\overset{(1),(2)}{=}&S_{x_{12}}S_{x_{23}}S_{x_{34}}S_{x_{13}}^{-1}S_{x_{11}}S_{x_{22}}S_{x_{33}x_{44}}S_{x_{34}}^{-1}\\
	&\overset{(2)}{=}&S_{x_{12}}S_{x_{23}}S_{x_{34}}S_{x_{13}}^{-1}S_{x_{11}}S_{x_{22}}S_{x_{33}}S_{x_{44}}S_{x_{34}}^{-1},
	\end{eqnarray*}
	\begin{eqnarray*}
	y^{-1}(S_{x_{22}})&=&S_{x_{23}x_{34}^{-1}x_{23}^{-1}x_{12}^{-1}x_{22}x_{34}x_{44}^{-1}x_{33}^{-1}x_{22}^{-1}x_{11}^{-1}x_{12}}\\
	&\overset{(2)}{=}&S_{x_{23}x_{42}}S_{x_{42}^{-1}x_{34}^{-1}x_{23}^{-1}x_{12}^{-1}x_{22}x_{34}x_{44}^{-1}x_{33}^{-1}x_{22}^{-1}x_{11}^{-1}x_{12}}\\
	&\overset{(2)}{=}&S_{x_{23}}S_{x_{42}}S_{x_{42}^{-1}x_{34}^{-1}x_{23}^{-1}x_{12}^{-1}x_{22}}S_{x_{34}x_{44}^{-1}x_{33}^{-1}x_{22}^{-1}x_{11}^{-1}x_{12}}\\
	&\overset{(1),(2)}{=}&S_{x_{23}}S_{x_{42}}S_{x_{42}}^{-1}S_{x_{34}}^{-1}S_{x_{23}^{-1}x_{12}^{-1}x_{22}}S_{x_{34}}S_{x_{44}^{-1}x_{33}^{-1}x_{22}^{-1}x_{11}^{-1}x_{12}}\\
	&\overset{(1),(2)}{=}&S_{x_{23}}S_{x_{34}}^{-1}S_{x_{23}^{-1}x_{12}^{-1}}S_{x_{22}}S_{x_{34}}S_{x_{44}}^{-1}S_{x_{33}}^{-1}S_{x_{22}^{-1}x_{11}^{-1}x_{12}}\\
	&\overset{(1),(2)}{=}&S_{x_{23}}S_{x_{34}}^{-1}S_{x_{23}}^{-1}S_{x_{12}}^{-1}S_{x_{22}}S_{x_{34}}S_{x_{44}}^{-1}S_{x_{33}}^{-1}S_{x_{22}^{-1}x_{11}^{-1}}S_{x_{12}}\\
	&\overset{(1),(2)}{=}&S_{x_{23}}S_{x_{34}}^{-1}S_{x_{23}}^{-1}S_{x_{12}}^{-1}S_{x_{22}}S_{x_{34}}S_{x_{44}}^{-1}S_{x_{33}}^{-1}S_{x_{22}}^{-1}S_{x_{11}}^{-1}S_{x_{12}},
	\end{eqnarray*}
	\begin{eqnarray*}
	y(S_{x_{33}})&=&S_{x_{34}x_{44}^{-1}x_{33}^{-1}x_{22}^{-1}x_{11}^{-1}x_{13}x_{23}^{-1}x_{22}x_{33}x_{34}^{-1}x_{23}^{-1}x_{12}^{-1}x_{23}}\\
	&\overset{(2)}{=}&S_{x_{34}x_{44}^{-1}x_{33}^{-1}x_{22}^{-1}x_{11}^{-1}x_{13}}S_{x_{23}^{-1}x_{22}x_{33}x_{34}^{-1}x_{23}^{-1}x_{12}^{-1}x_{23}}\\
	&\overset{(2)}{=}&S_{x_{34}}S_{x_{44}^{-1}x_{33}^{-1}x_{22}^{-1}x_{11}^{-1}x_{13}}S_{x_{23}^{-1}x_{22}x_{33}}S_{x_{34}^{-1}x_{23}^{-1}x_{12}^{-1}x_{23}}\\
	&\overset{(1),(2)}{=}&S_{x_{34}}S_{x_{44}}^{-1}S_{x_{33}^{-1}x_{22}^{-1}x_{11}^{-1}x_{13}}S_{x_{23}}^{-1}S_{x_{22}x_{33}}S_{x_{34}}^{-1}S_{x_{23}}^{-1}S_{x_{12}}^{-1}S_{x_{23}}\\
	&\overset{(2)}{=}&S_{x_{34}}S_{x_{44}}^{-1}S_{x_{33}^{-1}x_{22}^{-1}x_{11}^{-1}}S_{x_{13}}S_{x_{23}}^{-1}S_{x_{22}}S_{x_{33}}S_{x_{34}}^{-1}S_{x_{23}}^{-1}S_{x_{12}}^{-1}S_{x_{23}}\\
	&\overset{(1),(2)}{=}&S_{x_{34}}S_{x_{44}}^{-1}S_{x_{33}}^{-1}S_{x_{22}}^{-1}S_{x_{11}}^{-1}S_{x_{13}}S_{x_{23}}^{-1}S_{x_{22}}S_{x_{33}}S_{x_{34}}^{-1}S_{x_{23}}^{-1}S_{x_{12}}^{-1}S_{x_{23}},
	\end{eqnarray*}
	\begin{eqnarray*}
	y^{-1}(S_{x_{33}})&=&S_{x_{12}^{-1}x_{11}x_{22}x_{33}x_{44}x_{24}^{-1}x_{12}x_{23}x_{34}}\\
	&\overset{(1),(2)}{=}&S_{x_{12}}^{-1}S_{x_{11}x_{22}x_{33}x_{44}x_{24}^{-1}x_{12}x_{23}x_{34}}\\
	&\overset{(2)}{=}&S_{x_{12}}^{-1}S_{x_{11}x_{22}x_{33}x_{44}x_{24}^{-1}}S_{x_{12}x_{23}x_{34}}\\
	&\overset{(2)}{=}&S_{x_{12}}^{-1}S_{x_{11}}S_{x_{22}x_{33}x_{44}x_{24}^{-1}}S_{x_{12}}S_{x_{23}}S_{x_{34}}\\
	&\overset{(1),(2)}{=}&S_{x_{12}}^{-1}S_{x_{11}}S_{x_{22}x_{33}x_{44}}S_{x_{24}}^{-1}S_{x_{12}}S_{x_{23}}S_{x_{34}}\\
	&\overset{(2)}{=}&S_{x_{12}}^{-1}S_{x_{11}}S_{x_{22}}S_{x_{33}}S_{x_{44}}S_{x_{24}}^{-1}S_{x_{12}}S_{x_{23}}S_{x_{34}},
	\end{eqnarray*}
	\begin{eqnarray*}
	y(S_{x_{44}})&=&S_{x_{23}^{-1}x_{12}x_{23}x_{34}x_{44}x_{34}^{-1}x_{12}^{-1}x_{11}x_{22}x_{33}x_{44}}\\
	&\overset{(1),(2)}{=}&S_{x_{23}}^{-1}S_{x_{12}x_{23}x_{34}x_{44}x_{34}^{-1}x_{12}^{-1}x_{11}x_{22}x_{33}x_{44}}\\
	&\overset{(2)}{=}&S_{x_{23}}^{-1}S_{x_{12}}S_{x_{23}}S_{x_{34}x_{44}x_{34}^{-1}x_{12}^{-1}x_{11}x_{22}x_{33}x_{44}}\\
	&\overset{(2)}{=}&S_{x_{23}}^{-1}S_{x_{12}}S_{x_{23}}S_{x_{34}x_{44}x_{34}^{-1}}S_{x_{12}^{-1}x_{11}x_{22}x_{33}x_{44}}\\
	&\overset{(1),(2),(3)}{=}&S_{x_{23}}^{-1}S_{x_{12}}S_{x_{23}}S_{x_{34}}S_{x_{44}}S_{x_{34}}^{-1}S_{x_{12}}^{-1}S_{x_{11}x_{22}x_{33}x_{44}}\\
	&\overset{(2)}{=}&S_{x_{23}}^{-1}S_{x_{12}}S_{x_{23}}S_{x_{34}}S_{x_{44}}S_{x_{34}}^{-1}S_{x_{12}}^{-1}S_{x_{11}}S_{x_{22}}S_{x_{33}}S_{x_{44}},
	\end{eqnarray*}
	\begin{eqnarray*}
	y^{-1}(S_{x_{44}})&=&S_{x_{34}^{-1}x_{23}^{-1}x_{12}^{-1}x_{24}x_{34}^{-1}x_{22}^{-1}x_{11}^{-1}x_{12}x_{34}}\\
	&\overset{(2)}{=}&S_{x_{34}^{-1}x_{23}^{-1}x_{12}^{-1}x_{24}}S_{x_{34}^{-1}x_{22}^{-1}x_{11}^{-1}x_{12}x_{34}}\\
	&\overset{(2)}{=}&S_{x_{34}^{-1}x_{23}^{-1}}S_{x_{12}^{-1}x_{24}}S_{x_{34}^{-1}x_{22}^{-1}x_{11}^{-1}}S_{x_{12}x_{34}}\\
	&\overset{(1),(2)}{=}&S_{x_{34}}^{-1}S_{x_{23}}^{-1}S_{x_{12}}^{-1}S_{x_{24}}S_{x_{34}}^{-1}S_{x_{22}}^{-1}S_{x_{11}}^{-1}S_{x_{12}}S_{x_{34}},
	\end{eqnarray*}
	\begin{eqnarray*}
	y(S_{z_k})&=&S_{x_{23}^{-1}x_{12}x_{23}x_{34}z_kx_{34}^{-1}x_{23}^{-1}x_{12}^{-1}x_{23}}
	\overset{(2)}{=}S_{x_{23}^{-1}x_{12}x_{23}x_{34}z_k}S_{x_{34}^{-1}x_{23}^{-1}x_{12}^{-1}x_{23}}\\
	&\overset{(1),(2)}{=}&S_{x_{23}^{-1}x_{12}x_{23}x_{34}}S_{z_k}S_{x_{34}}^{-1}S_{x_{23}}^{-1}S_{x_{12}}^{-1}S_{x_{23}}\\
	&\overset{(1),(2)}{=}&S_{x_{23}}^{-1}S_{x_{12}}S_{x_{23}}S_{x_{34}}S_{z_k}S_{x_{34}}^{-1}S_{x_{23}}^{-1}S_{x_{12}}^{-1}S_{x_{23}},
	\end{eqnarray*}
	\begin{eqnarray*}
	y^{-1}(S_{z_k})&=&S_{x_{34}^{-1}x_{23}^{-1}x_{12}^{-1}x_{23}z_kx_{23}^{-1}x_{12}x_{23}x_{34}}
	\overset{(2)}{=}S_{x_{34}^{-1}x_{23}^{-1}x_{12}^{-1}}S_{x_{23}z_kx_{23}^{-1}x_{12}x_{23}x_{34}}\\
	&\overset{(1),(2)}{=}&S_{x_{34}}^{-1}S_{x_{23}}^{-1}S_{x_{12}}^{-1}S_{x_{23}}S_{z_k}S_{x_{23}}^{-1}S_{x_{12}}S_{x_{23}}S_{x_{34}}
	\end{eqnarray*}
	for $1\leq{k}\leq{n-1}$ only if $g=4$, and $y^{\pm1}(x)=x$ for any other $x\in{X^\prime}$.
	
	Let $y=Y_{\mu,a_1}$.
	We calculate
	\begin{eqnarray*}
	y(S_{x_{12}})&=&S_{x_{11}x_{22}x_{12}^{-1}}
	\overset{(1),(2)}{=}S_{x_{11}}S_{x_{22}}S_{x_{12}}^{-1},
	\end{eqnarray*}
	\begin{eqnarray*}
	y^{-1}(S_{x_{12}})&=&S_{x_{12}^{-1}x_{11}x_{22}}
	\overset{(1),(2)}{=}S_{x_{12}}^{-1}S_{x_{11}}S_{x_{22}},
	\end{eqnarray*}
	\begin{eqnarray*}
	y(S_{x_{23}})&=&S_{x_{11}x_{23}}
	\overset{(2)}{=}S_{x_{11}}S_{x_{23}},
	\end{eqnarray*}
	\begin{eqnarray*}
	y^{-1}(S_{x_{23}})&=&S_{x_{12}^{-1}x_{11}x_{12}x_{23}}
	\overset{(2)}{=}S_{x_{12}^{-1}x_{11}x_{12}}S_{x_{23}}
	\overset{(1),(3)}{=}S_{x_{12}}^{-1}S_{x_{11}}S_{x_{12}}S_{x_{23}},
	\end{eqnarray*}
	\begin{eqnarray*}
	y(S_{x_{11}})&=&S_{x_{11}x_{22}x_{12}^{-1}x_{11}^{-1}x_{12}x_{22}^{-1}x_{11}^{-1}}
	\overset{(2)}{=}S_{x_{11}x_{22}}S_{x_{12}^{-1}x_{11}^{-1}x_{12}x_{22}^{-1}x_{11}^{-1}}\\
	&\overset{(2)}{=}&S_{x_{11}}S_{x_{22}}S_{x_{12}^{-1}x_{11}^{-1}x_{12}}S_{x_{22}^{-1}x_{11}^{-1}}\\
	&\overset{(1),(2),(3)}{=}&S_{x_{11}}S_{x_{22}}S_{x_{12}}^{-1}S_{x_{11}}^{-1}S_{x_{12}}S_{x_{22}}^{-1}S_{x_{11}}^{-1},
	\end{eqnarray*}
	\begin{eqnarray*}
	y^{-1}(S_{x_{11}})&=&S_{x_{12}^{-1}x_{11}^{-1}x_{12}}
	\overset{(1),(3)}{=}S_{x_{12}}^{-1}S_{x_{11}}^{-1}S_{x_{12}},
	\end{eqnarray*}
	\begin{eqnarray*}
	y(S_{x_{22}})&=&S_{x_{11}x_{22}x_{12}^{-1}x_{11}x_{12}}
	\overset{(2)}{=}S_{x_{11}x_{22}}S_{x_{12}^{-1}x_{11}x_{12}}\\
	&\overset{(1),(2),(3)}{=}&S_{x_{11}}S_{x_{22}}S_{x_{12}}^{-1}S_{x_{11}}S_{x_{12}},
	\end{eqnarray*}
	\begin{eqnarray*}
	y^{-1}(S_{x_{22}})&=&S_{x_{12}^{-1}x_{11}x_{12}x_{11}x_{22}}
	\overset{(2)}{=}S_{x_{12}^{-1}x_{11}x_{12}}S_{x_{11}x_{22}}\\
	&\overset{(1),(2),(3)}{=}&S_{x_{12}}^{-1}S_{x_{11}}S_{x_{12}}S_{x_{11}}S_{x_{22}},
	\end{eqnarray*}
	\begin{eqnarray*}
	y(S_{z_k})&=&S_{x_{11}z_kx_{11}^{-1}}
	\overset{(1),(2)}{=}S_{x_{11}}S_{z_k}S_{x_{11}}^{-1}
	\end{eqnarray*}
	\begin{eqnarray*}
	y^{-1}(S_{z_k})&=&S_{x_{12}^{-1}x_{11}x_{12}z_kx_{12}^{-1}x_{11}^{-1}x_{12}}
	\overset{(2)}{=}S_{x_{12}^{-1}x_{11}x_{12}z_k}S_{x_{12}^{-1}x_{11}^{-1}x_{12}}\\
	&\overset{(1),(2),(3)}{=}&S_{x_{12}^{-1}x_{11}x_{12}}S_{z_k}S_{x_{12}}^{-1}S_{x_{11}}^{-1}S_{x_{12}}\\
	&\overset{(1),(3)}{=}&S_{x_{12}}^{-1}S_{x_{11}}S_{x_{12}}S_{z_k}S_{x_{12}}^{-1}S_{x_{11}}^{-1}S_{x_{12}},
	\end{eqnarray*}
	for $1\leq{k}\leq{n-1}$ only if $g=2$, and $y^{\pm1}(x)=x$ for any other $x\in{X^\prime}$.
	
	Let $y=B_{r_k}$ for $1\leq{k}\leq{n}$.
	We calculate
	\begin{eqnarray*}
	y(S_{x_{g-1\:g}})&=&S_{x_{g-1\:g}z_k}
	\overset{(2)}{=}S_{x_{g-1\:g}}S_{z_k},
	\end{eqnarray*}
	\begin{eqnarray*}
	y^{-1}(S_{x_{g-1\:g}})&=&S_{x_{g-1\:g}y_k}
	\overset{(2)}{=}S_{x_{g-1\:g}}S_{y_k},
	\end{eqnarray*}
	\begin{eqnarray*}
	y(S_{x_{gg}})&=&S_{x_{gg}y_kz_k}
	\overset{(2)}{=}S_{x_{gg}y_k}S_{z_k}
	\overset{(2)}{=}S_{x_{gg}}S_{y_k}S_{z_k},
	\end{eqnarray*}
	\begin{eqnarray*}
	y^{-1}(S_{x_{gg}})&=&S_{z_kx_{gg}y_k}
	\overset{(2)}{=}S_{z_k}S_{x_{gg}}S_{y_k},
	\end{eqnarray*}
	\begin{eqnarray*}
	y(S_{y_k})&=&S_{z_k^{-1}}
	\overset{(1)}{=}S_{z_k}^{-1},
	\end{eqnarray*}
	\begin{eqnarray*}
	y^{-1}(S_{y_k})&=&S_{y_k^{-1}x_{gg}^{-1}z_k^{-1}x_{gg}y_k}
	\overset{(2)}{=}S_{y_k^{-1}x_{gg}^{-1}z_k^{-1}}S_{x_{gg}y_k}\\
	&\overset{(1),(2)}{=}&S_{y_k^{-1}x_{gg}^{-1}}S_{z_k}^{-1}S_{x_{gg}}S_{y_k}
	\overset{(1),(2)}{=}S_{y_k}^{-1}S_{x_{gg}}^{-1}S_{z_k}^{-1}S_{x_{gg}}S_{y_k},
	\end{eqnarray*}
	\begin{eqnarray*}
	y(S_{y_m})&=&S_{z_k^{-1}y_k^{-1}y_my_kz_k}
	\overset{(2)}{=}S_{z_k^{-1}y_k^{-1}y_my_k}S_{z_k}
	\overset{(1),(2)}{=}S_{z_k}^{-1}S_{y_k}^{-1}S_{y_m}S_{y_k}S_{z_k},
	\end{eqnarray*}
	\begin{eqnarray*}
	y^{-1}(S_{y_m})&=&S_{y_k^{-1}x_{gg}^{-1}z_k^{-1}x_{gg}y_mx_{gg}^{-1}z_kx_{gg}y_k}
	\overset{(2)}{=}S_{y_k^{-1}x_{gg}^{-1}z_k^{-1}}S_{x_{gg}y_mx_{gg}^{-1}z_kx_{gg}y_k}\\
	&\overset{(1),(2)}{=}&S_{y_k}^{-1}S_{x_{gg}}^{-1}S_{z_k}^{-1}S_{x_{gg}y_mx_{gg}^{-1}}S_{z_kx_{gg}y_k}\\
	&\overset{(1),(2)}{=}&S_{y_k}^{-1}S_{x_{gg}}^{-1}S_{z_k}^{-1}S_{x_{gg}y_m}S_{x_{gg}}^{-1}S_{z_k}S_{x_{gg}}S_{y_k}\\
	&\overset{(2)}{=}&S_{y_k}^{-1}S_{x_{gg}}^{-1}S_{z_k}^{-1}S_{x_{gg}}S_{y_m}S_{x_{gg}}^{-1}S_{z_k}S_{x_{gg}}S_{y_k},
	\end{eqnarray*}
	\begin{eqnarray*}
	y(S_{z_k})&=&S_{x_{gg}y_k^{-1}x_{gg}^{-1}}
	\overset{(1),(2)}{=}S_{x_{gg}}S_{y_k}^{-1}S_{x_{gg}}^{-1},
	\end{eqnarray*}
	\begin{eqnarray*}
	y^{-1}(S_{z_k})&=&S_{y_k^{-1}}
	\overset{(1)}{=}S_{y_k}^{-1},
	\end{eqnarray*}
	\begin{eqnarray*}
	y(S_{z_m})&=&S_{z_k^{-1}z_mz_k}
	\overset{(1),(2)}{=}S_{z_k}^{-1}S_{z_m}S_{z_k},
	\end{eqnarray*}
	\begin{eqnarray*}
	y^{-1}(S_{z_m})&=&S_{y_k^{-1}z_my_k}
	\overset{(2)}{=}S_{y_k^{-1}z_m}S_{y_k}
	\overset{(1),(2)}{=}S_{y_k}^{-1}S_{z_m}S_{y_k},
	\end{eqnarray*}
	\begin{eqnarray*}
	y(S_{z_{m^\prime}})&=&S_{x_{gg}y_kx_{gg}^{-1}z_{m^\prime}x_{gg}y_k^{-1}x_{gg}^{-1}}
	\overset{(2)}{=}S_{x_{gg}y_kx_{gg}^{-1}}S_{z_{m^\prime}x_{gg}y_k^{-1}x_{gg}^{-1}}\\
	&\overset{(1),(2)}{=}&S_{x_{gg}y_k}S_{x_{gg}}^{-1}S_{z_{m^\prime}}S_{x_{gg}}S_{y_k}^{-1}S_{x_{gg}}^{-1}
	\overset{(2)}{=}S_{x_{gg}}S_{y_k}S_{x_{gg}}^{-1}S_{z_{m^\prime}}S_{x_{gg}}S_{y_k}^{-1}S_{x_{gg}}^{-1},
	\end{eqnarray*}
	\begin{eqnarray*}
	y^{-1}(S_{z_{m^\prime}})&=&S_{z_kz_{m^\prime}z_k^{-1}}
	\overset{(1),(2)}{=}S_{z_kz_{m^\prime}}S_{z_k}^{-1}
	\overset{(1),(2)}{=}S_{z_k}S_{z_{m^\prime}}S_{z_k}^{-1}
	\end{eqnarray*}
	for $1\leq{m}<k$ and $k<m^\prime\leq{n-1}$, and $y^{\pm1}(x)=x$ for any other $x\in{X^\prime}$.
	
	Let $y=B_{r_0}$.
	We calculate
	\begin{eqnarray*}
	y(S_{x_{i\:i+1}})&=&S_{x_{ig}^{-1}x_{ii}x_{i+1\:g}}
	\overset{(1),(2)}{=}S_{x_{ig}}^{-1}S_{x_{ii}}S_{x_{i+1\:g}},
	\end{eqnarray*}
	\begin{eqnarray*}
	y^{-1}(S_{x_{i\:i+1}})&=&S_{x_{gi}x_{i+1\:i+1}x_{g\:i+1}^{-1}}
	\overset{(1),(2)}{=}S_{x_{gi}x_{i+1\:i+1}}S_{x_{g\:i+1}}^{-1}\\
	&\overset{(2)}{=}&S_{x_{gi}}S_{x_{i+1\:i+1}}S_{x_{g\:i+1}}^{-1},
	\end{eqnarray*}
	\begin{eqnarray*}
	y(S_{x_{g-1\:g}})&=&S_{x_{g-1\:g}^{-1}x_{g-1\:g-1}x_{gg}}
	\overset{(1),(2)}{=}S_{x_{g-1\:g}}^{-1}S_{x_{g-1\:g-1}}S_{x_{gg}},
	\end{eqnarray*}
	\begin{eqnarray*}
	y^{-1}(S_{x_{g-1\:g}})&=&S_{x_{g\:g-1}},
	\end{eqnarray*}
	\begin{eqnarray*}
	y(S_{x_{jj}})&=&S_{x_{jg}^{-1}x_{jj}x_{jg}}
	\overset{(1),(3)}{=}S_{x_{jg}}^{-1}S_{x_{jj}}S_{x_{jg}},
	\end{eqnarray*}
	\begin{eqnarray*}
	y^{-1}(S_{x_{jj}})&=&S_{x_{gj}x_{jj}x_{gj}^{-1}}
	\overset{(1),(3)}{=}S_{x_{gj}}S_{x_{jj}}S_{x_{gj}}^{-1},
	\end{eqnarray*}
	\begin{eqnarray*}
	y^{\pm1}(S_{x_{gg}})&=&S_{x_{gg}},
	\end{eqnarray*}
	\begin{eqnarray*}
	y(S_{y_l})&=&S_{x_{gg}^{-1}z_lx_{gg}}
	\overset{(1),(2)}{=}S_{x_{gg}}^{-1}S_{z_l}S_{x_{gg}},
	\end{eqnarray*}
	\begin{eqnarray*}
	y^{-1}(S_{y_l})&=&S_{z_l},
	\end{eqnarray*}
	\begin{eqnarray*}
	y(S_{z_l})&=&S_{y_l},
	\end{eqnarray*}
	\begin{eqnarray*}
	y^{-1}(S_{z_l})&=&S_{x_{gg}y_lx_{gg}^{-1}}
	\overset{(1),(2)}{=}S_{x_{gg}y_l}S_{x_{gg}}^{-1}
	\overset{(2)}{=}S_{x_{gg}}S_{y_l}S_{x_{gg}}^{-1}
	\end{eqnarray*}
	for $1\leq{i}\leq{g-2}$, $1\leq{j}\leq{g-1}$ and $1\leq{l}\leq{n-1}$.
	
	Let $y=t_{s_{kl}}$ for $1\leq{k<l}\leq{n}$.
	We calculate
	\begin{eqnarray*}
	y(S_{y_k})&=&S_{y_ky_ly_ky_l^{-1}y_k^{-1}}
	\overset{(2)}{=}S_{y_ky_l}S_{y_ky_l^{-1}y_k^{-1}}
	\overset{(1),(2)}{=}S_{y_k}S_{y_l}S_{y_k}S_{y_l}^{-1}S_{y_k}^{-1},
	\end{eqnarray*}
	\begin{eqnarray*}
	y^{-1}(S_{y_k})&=&S_{y_l^{-1}y_ky_l}
	\overset{(1),(2)}{=}S_{y_l}^{-1}S_{y_k}S_{y_l},
	\end{eqnarray*}
	\begin{eqnarray*}
	y(S_{y_l})&=&S_{y_ky_ly_k^{-1}}
	\overset{(1),(2)}{=}S_{y_ky_l}S_{y_k}^{-1}
	\overset{(2)}{=}S_{y_k}S_{y_l}S_{y_k}^{-1},
	\end{eqnarray*}
	\begin{eqnarray*}
	y^{-1}(S_{y_l})&=&S_{y_l^{-1}y_k^{-1}y_ly_ky_l}
	\overset{(2)}{=}S_{y_l^{-1}y_k^{-1}y_l}S_{y_ky_l}\\
	&\overset{(2)}{=}&S_{y_l^{-1}y_k^{-1}}S_{y_l}S_{y_k}S_{y_l}
	\overset{(1),(2)}{=}S_{y_l}^{-1}S_{y_k}^{-1}S_{y_l}S_{y_k}S_{y_l},
	\end{eqnarray*}
	\begin{eqnarray*}
	y(S_{y_m})&=&S_{y_ky_ly_k^{-1}y_l^{-1}y_my_ly_ky_l^{-1}y_k^{-1}}
	\overset{(2)}{=}S_{y_ky_ly_k^{-1}y_l^{-1}y_my_l}S_{y_ky_l^{-1}y_k^{-1}}\\
	&\overset{(1),(2)}{=}&S_{y_ky_ly_k^{-1}}S_{y_l^{-1}y_my_l}S_{y_k}S_{y_l}^{-1}S_{y_k}^{-1}\\
	&\overset{(1),(2)}{=}&S_{y_ky_l}S_{y_k}^{-1}S_{y_l}^{-1}S_{y_m}S_{y_l}S_{y_k}S_{y_l}^{-1}S_{y_k}^{-1}\\
	&\overset{(2)}{=}&S_{y_k}S_{y_l}S_{y_k}^{-1}S_{y_l}^{-1}S_{y_m}S_{y_l}S_{y_k}S_{y_l}^{-1}S_{y_k}^{-1},
	\end{eqnarray*}
	\begin{eqnarray*}
	y^{-1}(S_{y_m})&=&S_{y_l^{-1}y_k^{-1}y_ly_ky_my_k^{-1}y_l^{-1}y_ky_l}
	\overset{(2)}{=}S_{y_l^{-1}y_k^{-1}y_l}S_{y_ky_my_k^{-1}y_l^{-1}y_ky_l}\\
	&\overset{(1),(2)}{=}&S_{y_l}^{-1}S_{y_k}^{-1}S_{y_l}S_{y_ky_my_k^{-1}}S_{y_l^{-1}y_ky_l}\\
	&\overset{(1),(2)}{=}&S_{y_l}^{-1}S_{y_k}^{-1}S_{y_l}S_{y_ky_m}S_{y_k}^{-1}S_{y_l}^{-1}S_{y_k}S_{y_l}\\
	&\overset{(2)}{=}&S_{y_l}^{-1}S_{y_k}^{-1}S_{y_l}S_{y_k}S_{y_m}S_{y_k}^{-1}S_{y_l}^{-1}S_{y_k}S_{y_l},
	\end{eqnarray*}
	\begin{eqnarray*}
	y(S_{z_k})&=&S_{z_kz_lz_kz_l^{-1}z_k^{-1}}
	\overset{(2)}{=}S_{z_kz_l}S_{z_kz_l^{-1}z_k^{-1}}
	\overset{(1),(2)}{=}S_{z_k}S_{z_l}S_{z_k}S_{z_l}^{-1}S_{z_k}^{-1},
	\end{eqnarray*}
	\begin{eqnarray*}
	y^{-1}(S_{z_k})&=&S_{z_l^{-1}z_kz_l}
	\overset{(1),(2)}{=}S_{z_l}^{-1}S_{z_k}S_{z_l},
	\end{eqnarray*}
	\begin{eqnarray*}
	y(S_{z_l})&=&S_{z_kz_lz_k^{-1}}
	\overset{(1),(2)}{=}S_{z_kz_l}S_{z_k}^{-1}
	\overset{(2)}{=}S_{z_k}S_{z_l}S_{z_k}^{-1},
	\end{eqnarray*}
	\begin{eqnarray*}
	y^{-1}(S_{z_l})&=&S_{z_l^{-1}z_k^{-1}z_lz_kz_l}
	\overset{(2)}{=}S_{z_l^{-1}z_k^{-1}z_l}S_{z_kz_l}\\
	&\overset{(2)}{=}&S_{z_l^{-1}z_k^{-1}}S_{z_l}S_{z_k}S_{z_l}
	\overset{(1),(2)}{=}S_{z_l}^{-1}S_{z_k}^{-1}S_{z_l}S_{z_k}S_{z_l},
	\end{eqnarray*}
	\begin{eqnarray*}
	y(S_{z_m})&=&S_{z_kz_lz_k^{-1}z_l^{-1}z_mz_lz_kz_l^{-1}z_k^{-1}}
	\overset{(2)}{=}S_{z_kz_lz_k^{-1}z_l^{-1}z_mz_l}S_{z_kz_l^{-1}z_k^{-1}}\\
	&\overset{(1),(2)}{=}&S_{z_kz_lz_k^{-1}}S_{z_l^{-1}z_mz_l}S_{z_k}S_{z_l}^{-1}S_{z_k}^{-1}\\
	&\overset{(1),(2)}{=}&S_{z_kz_l}S_{z_k}^{-1}S_{z_l}^{-1}S_{z_m}S_{z_l}S_{z_k}S_{z_l}^{-1}S_{z_k}^{-1}\\
	&\overset{(2)}{=}&S_{z_k}S_{z_l}S_{z_k}^{-1}S_{z_l}^{-1}S_{z_m}S_{z_l}S_{z_k}S_{z_l}^{-1}S_{z_k}^{-1},
	\end{eqnarray*}
	\begin{eqnarray*}
	y^{-1}(S_{z_m})&=&S_{z_l^{-1}z_k^{-1}z_lz_kz_mz_k^{-1}z_l^{-1}z_kz_l}
	\overset{(2)}{=}S_{z_l^{-1}z_k^{-1}z_l}S_{z_kz_mz_k^{-1}z_l^{-1}z_kz_l}\\
	&\overset{(1),(2)}{=}&S_{z_l}^{-1}S_{z_k}^{-1}S_{z_l}S_{z_kz_mz_k^{-1}}S_{z_l^{-1}z_kz_l}\\
	&\overset{(1),(2)}{=}&S_{z_l}^{-1}S_{z_k}^{-1}S_{z_l}S_{z_kz_m}S_{z_k}^{-1}S_{z_l}^{-1}S_{z_k}S_{z_l}\\
	&\overset{(2)}{=}&S_{z_l}^{-1}S_{z_k}^{-1}S_{z_l}S_{z_k}S_{z_m}S_{z_k}^{-1}S_{z_l}^{-1}S_{z_k}S_{z_l}
	\end{eqnarray*}
	for $k<m<l$, and $y^{\pm1}(x)=x$ for any other $x\in{X^\prime}$.
	
	Hence we have that for any $x\in{X^\prime}$ and $y\in{Y}^{\pm1}$, $y(x)$ is in the subgroup of $\pi^+$ generated by $X^\prime$, by Lemma~\ref{ij}
\end{proof}
%%%%%%%%%%%%%%%%%%%%%%%%%%%%%%%%%%%%%%%%%%%%%%%%%%%%%%%%%%%%%%%%%%%%%%%%%%%%%%%%%%%%%%%%%%%%%%%%%%%%
%%%%%%%%%%%%%%%%%%%%%%%%%%%%%%%%%%%%%%%%%%%%%%%%%%%%%%%%%%%%%%%%%%%%%%%%%%%%%%%%%%%%%%%%%%%%%%%%%%%%

%%%%%%%%%%%%%%%%%%%%%%%%%%%%%%%%%%%%%%%%%%%%%%%%%%%%%%%%%%%%%%%%%%%%%%%%%%%%%%%%%%%%%%%%%%%%%%%%%%%%
%%%%%%%%%%%%%%%%%%%%%%%%%%%%%%%%%%%%%%%%%%%%%%%%%%%%%%%%%%%%%%%%%%%%%%%%%%%%%%%%%%%%%%%%%%%%%%%%%%%%
\begin{proof}[Proof of Theorem~\ref{main-2}]
By Lemma~\ref{main-lem} and Proposition~\ref{3}, it follows that $\pi^+$ is generated by $X^\prime$.
There is a natural map $\pi^+\to\pi_1^+(N_{g,n},\ast)$.
The relations~(1), (2) and (3) of $\pi^+$ are satisfied in $\pi_1^+(N_{g,n},\ast)$ clearly.
Hence the map is a homomorphism.
In addition, the relations $S_{x_{11}}\cdots{}S_{x_{gg}}=1$ and $S_{x_{gg}}S_{x_{g-1\:g}}^{-1}S_{x_{g-1\:g-1}}S_{x_{g-2\:g-1}}^{-1}\cdots{}S_{x_{22}}S_{x_{12}}^{-1}S_{x_{11}}S_{x_{12}}\cdots{}S_{x_{g-1\:g}}=1$ are obtained from the relations~(1), (2) and (3) of $\pi^+$ for $n=0$.
Therefore the map is an isomorphism for any $n\geq0$.
Thus we complete the proof.
\end{proof}
%%%%%%%%%%%%%%%%%%%%%%%%%%%%%%%%%%%%%%%%%%%%%%%%%%%%%%%%%%%%%%%%%%%%%%%%%%%%%%%%%%%%%%%%%%%%%%%%%%%%
%%%%%%%%%%%%%%%%%%%%%%%%%%%%%%%%%%%%%%%%%%%%%%%%%%%%%%%%%%%%%%%%%%%%%%%%%%%%%%%%%%%%%%%%%%%%%%%%%%%%

%%%%%%%%%%%%%%%%%%%%%%%%%%%%%%%%%%%%%%%%%%%%%%%%%%%%%%%%%%%%%%%%%%%%%%%%%%%%%%%%%%%%%%%%%%%%%%%%%%%%
%%%%%%%%%%%%%%%%%%%%%%%%%%%%%%%%%%%%%%%%%%%%%%%%%%%%%%%%%%%%%%%%%%%%%%%%%%%%%%%%%%%%%%%%%%%%%%%%%%%%
\begin{rem}\label{reduced-rel}
Simple loops $\alpha$, $\beta$ and $\gamma$ of the relation~(3) in Theorem~\ref{main-2} can be reduced to the form as shown in Figure~\ref{rel-3}.
In fact, we used only this reduced relation as the relation~(3), in the proof of Theorem~\ref{main-2}.
We do not know whether the relation~(3) in Theorem~\ref{main-2} can be obtained from the relations~(1) and (2) there or not.
%%%%%%%%%%%%%%%%%%%%%%%%%%%%%%%%%%%%%%%%%%%%%%%%%%%%%%%%%%%%%%%%%%%%%%%%%%%%%%%%%%%%%%%%%%%%%%%%%%%%
\begin{figure}[htbp]
\includegraphics[scale=0.5]{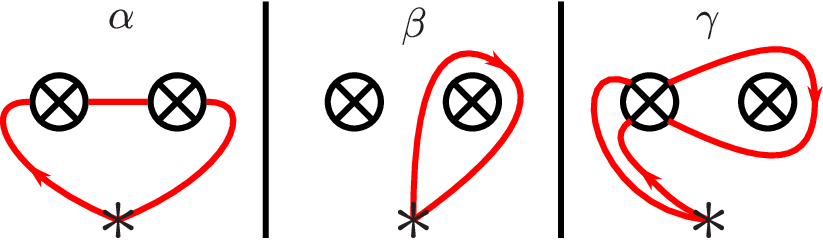}
\caption{The reduced relation $S_{\alpha}S_{\beta}S_{\alpha}^{-1}=S_\gamma$.}\label{rel-3}
\end{figure}
%%%%%%%%%%%%%%%%%%%%%%%%%%%%%%%%%%%%%%%%%%%%%%%%%%%%%%%%%%%%%%%%%%%%%%%%%%%%%%%%%%%%%%%%%%%%%%%%%%%%
\end{rem}
%%%%%%%%%%%%%%%%%%%%%%%%%%%%%%%%%%%%%%%%%%%%%%%%%%%%%%%%%%%%%%%%%%%%%%%%%%%%%%%%%%%%%%%%%%%%%%%%%%%%
%%%%%%%%%%%%%%%%%%%%%%%%%%%%%%%%%%%%%%%%%%%%%%%%%%%%%%%%%%%%%%%%%%%%%%%%%%%%%%%%%%%%%%%%%%%%%%%%%%%%

%%%%%%%%%%%%%%%%%%%%%%%%%%%%%%%%%%%%%%%%%%%%%%%%%%%%%%%%%%%%%%%%%%%%%%%%%%%%%%%%%%%%%%%%%%%%%%%%%%%%
%%%%%%%%%%%%%%%%%%%%%%%%%%%%%%%%%%%%%%%%%%%%%%%%%%%%%%%%%%%%%%%%%%%%%%%%%%%%%%%%%%%%%%%%%%%%%%%%%%%%
%%%%%%%%%%%%%%%%%%%%%%%%%%%%%%%%%%%%%%%%%%%%%%%%%%%%%%%%%%%%%%%%%%%%%%%%%%%%%%%%%%%%%%%%%%%%%%%%%%%%
%%%%%%%%%%%%%%%%%%%%%%%%%%%%%%%%%%%%%%%%%%%%%%%%%%%%%%%%%%%%%%%%%%%%%%%%%%%%%%%%%%%%%%%%%%%%%%%%%%%%
%%%%%%%%%%%%%%%%%%%%%%%%%%%%%%%%%%%%%%%%%%%%%%%%%%%%%%%%%%%%%%%%%%%%%%%%%%%%%%%%%%%%%%%%%%%%%%%%%%%%
%%%%%%%%%%%%%%%%%%%%%%%%%%%%%%%%%%%%%%%%%%%%%%%%%%%%%%%%%%%%%%%%%%%%%%%%%%%%%%%%%%%%%%%%%%%%%%%%%%%%
%%%%%%%%%%%%%%%%%%%%%%%%%%%%%%%%%%%%%%%%%%%%%%%%%%%%%%%%%%%%%%%%%%%%%%%%%%%%%%%%%%%%%%%%%%%%%%%%%%%%
%%%%%%%%%%%%%%%%%%%%%%%%%%%%%%%%%%%%%%%%%%%%%%%%%%%%%%%%%%%%%%%%%%%%%%%%%%%%%%%%%%%%%%%%%%%%%%%%%%%%
%%%%%%%%%%%%%%%%%%%%%%%%%%%%%%%%%%%%%%%%%%%%%%%%%%%%%%%%%%%%%%%%%%%%%%%%%%%%%%%%%%%%%%%%%%%%%%%%%%%%
%%%%%%%%%%%%%%%%%%%%%%%%%%%%%%%%%%%%%%%%%%%%%%%%%%%%%%%%%%%%%%%%%%%%%%%%%%%%%%%%%%%%%%%%%%%%%%%%%%%%
\section*{Acknowledgement}

The author would like to express his thanks to Andrew Putman and Masatoshi Sato for their useful comments.

%%%%%%%%%%%%%%%%%%%%%%%%%%%%%%%%%%%%%%%%%%%%%%%%%%%%%%%%%%%%%%%%%%%%%%%%%%%%%%%%%%%%%%%%%%%%%%%%%%%%
%%%%%%%%%%%%%%%%%%%%%%%%%%%%%%%%%%%%%%%%%%%%%%%%%%%%%%%%%%%%%%%%%%%%%%%%%%%%%%%%%%%%%%%%%%%%%%%%%%%%
%%%%%%%%%%%%%%%%%%%%%%%%%%%%%%%%%%%%%%%%%%%%%%%%%%%%%%%%%%%%%%%%%%%%%%%%%%%%%%%%%%%%%%%%%%%%%%%%%%%%
%%%%%%%%%%%%%%%%%%%%%%%%%%%%%%%%%%%%%%%%%%%%%%%%%%%%%%%%%%%%%%%%%%%%%%%%%%%%%%%%%%%%%%%%%%%%%%%%%%%%
%%%%%%%%%%%%%%%%%%%%%%%%%%%%%%%%%%%%%%%%%%%%%%%%%%%%%%%%%%%%%%%%%%%%%%%%%%%%%%%%%%%%%%%%%%%%%%%%%%%%

\end{document}